\documentclass[a4paper,12pt,reqno]{amsart}
\allowdisplaybreaks
\usepackage{amsmath,amssymb}%
\usepackage{exscale}%
\usepackage{amsfonts}%
\usepackage{palatino}
\usepackage{amsthm}
\usepackage{epsfig}
\usepackage{xcolor}
\usepackage{hyperref}
\usepackage[utf8]{inputenc}
\usepackage[english]{babel}
\usepackage{mathtools}
\usepackage{cite}

\usepackage{xcolor}
\usepackage[ left=2.8cm,
right=2.8cm,
top=3.5cm,
bottom=3.5cm]{geometry}

\usepackage{cleveref}
\usepackage[footnote,draft,silent,nomargin]{fixme}

\hypersetup{%
	pdfpagelabels=true,%
	plainpages=false,%
	pdfauthor={Author(s)},%
	pdftitle={Title},%
	pdfsubject={Subject},%
	bookmarksnumbered=true,%
	colorlinks,%
	citecolor=black,%
	filecolor=black,%
	linkcolor= black,% you should probably change this to black before printing
	urlcolor=black,%
	pdfstartview=FitH%
}

\let\epsilon\varepsilon

\let\phi\varphi

\let\theta\vartheta

\newtheorem{mytheorem}{Theorem}[section]
\newtheorem{myprop}[mytheorem]{Proposition}
\newtheorem{mycor}[mytheorem]{Corollary} 
\theoremstyle{definition}
\newtheorem{mydef}[mytheorem]{Definition}
\newtheorem{myex}[mytheorem]{Example}
\newtheorem{myre}[mytheorem]{Remark}
\newtheorem{mylemma}[mytheorem]{Lemma}

\newcommand{\C}{\mathbb{C}}
\newcommand{\R}{\mathbb{R}}
\newcommand{\Z}{\mathbb{Z}}
\newcommand{\N}{\mathbb{N}}

\newcommand{\ca}{\mathbf{a}}

\newcommand{\ip}[1]{{\left\langle #1\right\rangle}}

\newcommand{\norm}[1]{\left\lVert #1\right\rVert}

\renewcommand{\vec}[1]{\mathbf{#1}}

\DeclareMathOperator{\e}{e}
\renewcommand{\d}{\mbox{d}}
\renewcommand{\epsilon}{\varepsilon}

\begin{document}
\title[On a discrete transform of homogeneous decomposition spaces]{On a discrete transform of homogeneous decomposition spaces } \author{Zeineb Al-Jawahri and Morten Nielsen }   \date{\today}
\begin{abstract}
We introduce almost diagonal matrices in the setting of (anisotropic) discrete homogeneous Triebel-Lizorkin type spaces and homogeneous modulation spaces, and it is shown that the class of almost diagonal matrices is closed under matrix multiplication.  

We then connect the results to the continuous setting and show that the  "change of frame" matrix for a pair of time-frequency frames, with suitable decay properties, is almost diagonal.  As an application of this result, we consider a construction of compactly supported frame expansions for homogeneous decomposition spaces of Triebel-Lizorkin type and for the associated modulation spaces. 
\end{abstract}
\maketitle

\section{Introduction}

Function spaces based on anisotropic Littlewood-Paley decompositions have attracted considerable interest in recent years, see for example \cite{BH05,MR3016517,MR3345605,MR3048591, MR3981281,MR3969420,MR3707993,MR4029179,BN08,MR3303346} and reference therein. This renewed interest in such spaces is to a large extent driven by advances in the study of partial and pseudo-differential operators, where there is a natural desire to be able to better model and analyse anisotropic phenomena. The connection to constructive algorithms suitable for applications and numerics is often made possible by considering suitable discretised sampled versions of the underlying Littlewood-Paley decomposition.

In the present paper we will study additional features of certain discrete representations of homogeneous decomposition smoothness spaces. The theory of decomposition spaces, introduced by  Feichtinger and  Gr\"obner \cite{MR809337} and by Feichtinger \cite{MR910054},  is an abstract  general machinery for building function spaces.  This machinery, when tuned to decompositions of the frequency domain,  covers a large range of smoothness spaces that have turned out to be of interest for applications.  The close connection between decomposition spaces and  classical smoothness space such as modulation spaces was  first pointed out by Triebel \cite{MR725159}. Triebel's work  later inspired a more general treatment of decomposition smoothness spaces  \cite{BN08,Nielsen}. In the same spirit, very general homogeneous (anisotropic) Besov and Triebel-Lizorkin spaces based on dyadic decompositions were considered by Bownik \cite{MR2179611} and by Bownik and Ho  \cite{BH05}. In a similar dyadic setup, a general approach to homogeneous spaces has been studied in detail recently by Triebel \cite{MR3409094,MR3682619}.

 The present authors considered a general construction of homogeneous smoothness spaces, based on structured decomposition of the frequency space $\mathbb{R}^d\backslash\{0\}$,  in \cite{AN19}. Adapted tight frames for $L_2(\mathbb{R}^d)$ were were considered in \cite{AN19} and they were shown to provide universal decompositions of tempered distributions with convergence in the tempered distributions modulo polynomials. Moreover, atomic decompositions of the corresponding homogeneous smoothness spaces were obtained, completely characterising the smoothness spaces by a sparseness condition on the frame coefficients, facilitating compression of the elements of such homogeneous smoothness spaces using the corresponding frame coefficients. An alternative approach to homogeneous decomposition type spaces based on the theory of Coorbit-spaces has been  considered by F\"uhr and Voigtlaender in \cite{MR3345605}.
 
 In the present paper, which can be considered a continuation of \cite{AN19}, we study additional properties of discrete representations of  homogeneous decomposition spaces of Besov and Triebel-Lizorkin type. Most importantly, in Section \ref{sec:almostdiagonal}, we introduce the notion of almost diagonal matrices for homogeneous decomposition spaces of Besov and Triebel-Lizorkin type and we use the tight frame introduced \cite{AN19} to link such matrices to bounded operators on Besov and Triebel-Lizorkin type spaces.
 
 The main contribution of the present paper is a detailed proof that "change of frame coefficient" matrices between any two suitably localised adapted time-fre-quency frames is almost diagonal.  This also leads to a natural definition of decomposition space molecules. 
 In the inhomogeneous setting, similar results were considered in \cite{NR12}. However, as it turns out, the homogeneous setup presents several additional challenges that will be addressed in this paper. The result can be found in Section \ref{sec:almostdiagonal}.
 
As an application of the results obtained, we study various perturbation of the  frame from \cite{AN19} to obtain  compactly supported frames for homogeneous decomposition spaces of Besov and Triebel-Lizorkin type. This is considered in Section \ref{sec:compactly}. Sections 2 and 3 contain the needed preliminary facts and results.

\section{Preliminaries}
We now introduce the notation needed to define and study homogeneous decomposition spaces. The terminology is to a large degree inherited from Feichtinger and Gr\"obner, see \cite{MR809337,MR910054}, and from \cite{Nielsen}. 

\subsection{Anisotropic geometry}
Let $|\cdot|$ denote the Euclidean norm on $\mathbb{R}^d$ induced by the standard inner product $\ip{\cdot,\cdot}$ and let $\vec{a} = (a_1,\ldots,a_d)\in \mathbb{R}_+^d$ be an anisotropy on $\mathbb{R}^d$ scaled such that $a_i > 1$ and $\sum_{i=1}^d a_i = \nu > d$.
For $t>0$, define the anisotropic dilation matrix as $D_{\vec{a}}(t) := \text{diag}(t^{a_1},\ldots,t^{a_d})$.
We mention that the specific scaling assumption for the anisotropy $\ca$ is chosen to facilitate certain technical estimates in Section \ref{sec:almostdiagonal}, see  Remark \ref{re:scale}.

\begin{mydef}\label{def:aninorm} 
	We define the function $|\cdot|_{\vec{a}} : \R^d\to \R_+$ by setting $|0|_{\vec{a}} := 0$ and for $\xi\in\R^d\backslash\{0\}$ we set $|\xi|_{\vec{a}} = t$, where $t$ is the unique solution to the equation $\left|D_{\vec{a}}(1/t)\xi\right| = 1$.
\end{mydef}

According to \cite{Stein} we have the following standard properties of $|\cdot|_{\vec{a}}$:

\begin{itemize}
	\item[(1)] $|\cdot|_{\vec{a}} \in C^\infty(\mathbb{R}^d\backslash\{0\})$.
	\item[(2)] There exists a constant $K\geq 1$ such that
	\begin{equation*}
	|\xi + \zeta|_{\vec{a}} \leq K(|\xi|_{\vec{a}}+|\zeta|_{\vec{a}}), \quad \forall \; \xi,\zeta\in\mathbb{R}^d\backslash\{0\}.
	\end{equation*}
	\item[(3)] For $t>0$,
	\begin{equation}\label{eq:conditionaninorm}
	|D_{\vec{a}}(t)\xi|_{\vec{a}} = t|\xi|_{\vec{a}}
	\end{equation}
	\item[(4)] For $\xi \in \mathbb{R}^d\backslash\{0\}$,
	\begin{align}
	c_1 |\xi|^{\alpha_1} &\leq |\xi|_{\vec{a}} \leq c_2 |\xi|^{\alpha_2}, \quad \text{if} \; |\xi|_{\vec{a}} \geq 1, \; \text{and} \nonumber \\
	 c_3 |\xi|^{\alpha_2} &\leq |\xi|_{\vec{a}} \leq c_4 |\xi|^{\alpha_1}, \quad \text{if} \; |\xi|_{\vec{a}} < 1, \label{eq:aniless1}
	\end{align}
	where $\alpha_1 := \min_{1\leq i \leq d} a_i$ and $\alpha_2:=\max_{1\leq i\leq d} a_i$.
\end{itemize}

The anisotropic norm $|\cdot|_{\vec{a}}$ from Definition \ref{def:aninorm} induces a quasi-distance $\d: \mathbb{R}^d \times \mathbb{R}^d \to [0,\infty)$ given by $\d(\xi,\zeta) := |\xi-\zeta|_{\vec{a}}$. The (anisotropic) ball of radius $r>0$ centered at $\xi\in\R^d$ is given by
$$B_\mathbf{a}(\xi,r):=\{\zeta\in \R^d:\d(\xi,\zeta)<r\}.$$
 One can verify that $(\mathbb{R}^d\backslash\{0\},\d,\d \xi)$ is a space of homogeneous type.

\subsection{Maximal operators}
Since we will study function spaces of Triebel-Lizorkin type,  maximal function estimates will play a central role.

For $0<r<\infty$, the maximal function of Hardy-Littlewood type is defined by
\begin{equation}\label{eq:maximaloperator}
M_r^{\ca}u(x) := \sup_{t>0} \left( \frac{1}{\kappa^{\ca}_d \cdot t^\nu} \int_{B_{\ca}(x,t)} |u(y)|^r\d y \right)^{\frac{1}{r}}, \quad u \in L_{r,\text{loc}}(\R^d),
\end{equation}
where $\kappa_d^{\ca} := |B_{\ca}(0,1)|$. Moreover, due to the structured anisotropic setup, we have 
the following vector-valued Fefferman-Stein maximal inequality, see \cite{MR1232192}.
For $0<r\leq q \leq \infty$, and $r<p<\infty$, there exists $C>0$ such that
\begin{equation*}\label{eq:feffermanstein}
\norm{ \left( \sum_{j\in J} |M_r^{\ca} f_j|^q\right)^{1/q}}_{L_p} \leq C \norm{\left( \sum_{j\in J} |f_j|^q\right)^{1/q}}_{L_p}.
\end{equation*}
If $q=\infty$, then the inner $\ell_q$-norm is replaced by the $\ell_\infty$-norm. \\

\section{Homogeneous Triebel-Lizorkin Type Spaces}
In this section we define homogeneous Triebel-Lizorkin (T-L)  and Besov type spaces. This is done by considering certain structured admissible coverings of the frequency space $\mathbb{R}^d\backslash\{0\}$. The coverings are used to construct a suitable resolution of unity which can be used define the T-L type spaces and the decomposition spaces. 

We simplify the construction in the sense that we use a suitable collection of $\d$-balls to cover $\mathbb{R}^d\backslash\{0\}$, where the radius of a given ball in the cover is a so-called hybrid regulation function. 

A simple construction of a tight-frame for the various homogeneous T-L  and Besov type spaces spaces is also considered. The particular frame will be shown to fully characterise  the (quasi-)norm on homogeneous T-L  and Besov type, and it will prove essential for our analysis of almost diagonal matrices in Section \ref{sec:almostdiagonal}. Let us recall the notion of a moderate function.
\begin{mydef}\label{dmod}
 A function $h:\mathbb{R}^d\backslash\{0\} \to (0,\infty)$ is called $\d$-moderate if there exist constants $R,\delta_0 > 0$ such that $\d(\xi,\zeta) \leq \delta_0 h(\xi)$ implies $R^{-1} \leq h(\xi)/h(\zeta) \leq R$ for all $\xi,\zeta\in\mathbb{R}^d\backslash\{0\}$. 
\end{mydef}

We now use a ramp function to glue two moderate functions together. The idea is to separately  ''regulate''  low frequencies and  high frequencies.

\begin{mydef} \label{def:hfunction}
	Take a non-negative ramp function $\rho \in \mathbb{C}^s$ for some $s\geq 1$ satisfying
\begin{equation}\label{eq:rho}	
\rho(\xi) = \begin{cases}
1& \text{for $0<|\xi|_{\vec{a}} \leq \frac{2}{3}$} \\ 
\\
0& \text{for $|\xi|_{\vec{a}} \geq \frac{4}{3}$}
\end{cases}
\end{equation}
and define $\tilde{h}: \mathbb{R}^d\backslash\{0\} \to (0,\infty)$ as
\begin{equation}\label{eq:hybrid}
\tilde{h}(\xi) = \rho h_1(\xi) + (1-\rho)h_2(\xi),
\end{equation}
where $h_1(\xi)$ and $h_2(\xi)$ are both $\d$-moderate functions satisfying
\begin{subequations}
	\begin{align}\label{eq:h1}
	c_0|\xi|_{\vec{a}}^{r} \leq h_1(\xi) &\leq c_1|\xi|_{\vec{a}}, &\text{for
		some $c_0, c_1>0$ and $r\geq 1$,} \\ \intertext{and}
	\label{eq:h2}c_2 \leq h_2(\xi) &\leq c_3|\xi|_{\vec{a}}, &\text{for some $c_2, c_3 > 0$.}
	\end{align}
\end{subequations}
We call $\tilde{h}: \mathbb{R}^d\backslash\{0\} \to (0,\infty)$ a \textit{hybrid regulation function}. 
\end{mydef}
We mention that according to \cite[Lemma 2.7]{AN19}, $\tilde{h}$ itself is a  $\d$-moderate function in the sense of Definition \ref{dmod}.

\begin{myex}\label{eks:hrf}
Let $\alpha \in [0,1]$. Then
\begin{equation*}
\tilde{h}(\xi) := \rho h_1(\xi)+(1-\rho)h_2(\xi),
\end{equation*}
where $\rho$ satisfies \eqref{eq:rho}, $h_1(\xi) := |\xi|_{\ca}^{2-\alpha}$ and $h_2(\xi) := |\xi|_{\ca}^\alpha$, is a hybrid regulation function.  
\end{myex}

With a hybrid regulation function $\tilde{h}$, we can construct a structured admissible covering by open (anisotropic) balls. 

\begin{mylemma}\label{Lemma5}
	Consider $(\mathbb{R}^d\backslash\{0\},\d,\d\xi)$ and let $\tilde{h}: \mathbb{R}^d\backslash\{0\} \to (0,\infty)$ be a hybrid regulation function. 
	Pick $0<\delta < \delta_0/2$. Then there exist 
	 an ordered countable (infinite) index set $J\ne \emptyset$ and 
	an admissible covering $\mathcal{Q}=\{B_{\mathbf{a}}(\xi_j,\delta \tilde{h}(\xi_j))\}_{j\in J}$ of $\mathbb{R}^d\backslash\{0\}$ and a constant $0<\delta'<\delta$ such that $\{B_{\mathbf{a}}(\xi_j,\delta'\tilde{h}(\xi_j))\}_{j\in J}$ are pairwise disjoint.
\end{mylemma}

Since the balls in the collection $\{B_{\mathbf{a}}(\xi_j,\delta'\tilde{h}(\xi_j))\}_{j\in J}$ are pairwise disjoint, it can be verified that $\{B_{\mathbf{a}}(\xi_j,2\delta \tilde{h}(\xi_j))\}_{j\in J}$ gives a structured admissible covering of $\mathbb{R}^d\backslash\{0\}$. 
We note that the covering $\mathcal{Q}$ from Lemma \ref{Lemma5} is generated by a family of invertible affine transformations applied to the $\d$-ball $Q:=B_{\mathbf{a}}(0,\delta)$. That is,
\begin{align}\label{eq:matrixA}
\{B_{\mathbf{a}}(\xi_j,\delta\tilde{h}(\xi_j))\}_{j\in J} := \{T_jQ\}_{j\in J}, \quad \text{where} \; T_jQ := A_j Q+\xi_j, \; \; A_j := D_{\vec{a}}(\tilde{h}(\xi_j)).
\end{align}	
The matrices $\{A_j\}_j$ and the frequencies $\{\xi_j\}_j$ will be kept fixed throughout the paper. \\

We can use the structured admissible covering $\mathcal{Q}$ from Lemma \ref{Lemma5} to generate a suitable resolution of unity that can be used to define the T-L type spaces and the decomposition spaces. Due to technical reasons we require the partiton of unity to satisfy the following. 

\begin{mydef}\label{def:bapu}
Let $\mathcal{Q} := \{Q_j\}_{j\in J} := \{T_jB_{\mathbf{a}}(0,\delta)\}_{j\in J}$ be a structured admissible covering of $\mathbb{R}^d\backslash\{0\}$. A corresponding bounded admissible partition of unity (BAPU) is a family of functions $\Psi = \{\psi_j\}_{j\in J} \subset \mathcal{S}(\mathbb{R}^d)$ satisfying
	\begin{itemize}
		\item[(1)] $\text{supp}(\psi_j) \subseteq T_jB_{\mathbf{a}}(0,2\delta)$ for all $j \in J$.
		\item[(2)] $\sum_{j\in J}\psi_j(\xi) = 1$ for all $\xi \in \mathbb{R}^d\backslash\{0\}$,
		\item[(3)] $\sup_{j\in J} \norm{\psi_j(T_j\cdot)}_{H_2^s} < \infty, s>0$, 
	\end{itemize}
where $\norm{f}_{H_2^s} := (\int |\mathcal{F}^{-1} f(x)|^2(1+|x|_{\ca})^{2s} \d x)^{1/2}$.
\end{mydef}

We use a standard trick for constructing a BAPU for $\mathcal{Q}$. Pick $\Psi \in \mathbb{C}^\infty(\mathbb{R}^d)$ non-negative with $\text{supp}(\Psi) \subseteq B_{\mathbf{a}}(0,2\delta)$ and $\Psi(\xi) = 1$ for $\xi \in B_{\mathbf{a}}(0,\delta)$. Then it can be shown that
\begin{align*}
\psi_j(\xi) := \frac{\Psi(T_j^{-1}\xi)}{\sum_{k\in J}\Psi(T_k^{-1}\xi)}
\end{align*}
defines a BAPU for $\mathcal{Q}$. 
For later use, we also introduce
\begin{equation}\label{eq:squarerootbapu}
\phi_j(\xi) := \frac{\Phi(T_j^{-1}\xi)}{\sqrt{\sum_{k\in J} \Phi(T_k^{-1}\xi)^2}},
\end{equation}
which in a sense defines a ''square root'' of a BAPU.

We can now define the homogeneous (anisotropic) T-L type spaces and the decomposition spaces. We let $ \mathcal{S}'\backslash\mathcal{P}$ denote the class of tempered distributions modulo polynomials defined on $\mathbb{R}^d$.

\begin{mydef}\label{def:complete}
Let $\tilde{h}$ be a hybrid regulation function and let $\mathcal{Q}$ be a structured admissible covering generated by $\tilde{h}$ of the type considered in Lemma \ref{Lemma5}. Let $\{\phi_j\}_{j\in J}$ be a corresponding BAPU and set $\phi_j(D)f := \mathcal{F}^{-1}(\phi_j\mathcal{F}f)$.
\begin{itemize}
	\item For $s\in\R, 0<p<\infty$ and $0<q\leq \infty$, we define the (anisotropic) homogeneous Triebel-Lizorkin space $\dot{F}_{p,q}^s(\tilde{h})$ as the set all $f\in \mathcal{S}'\backslash\mathcal{P}$ satisfying
	\begin{equation*}
	\norm{f}_{\dot{F}_{p,q}^s(\tilde{h})} := \norm{\left( \sum_{j\in J} |\tilde{h}(\xi_j)^s\phi_j(D)f|^q\right)^{1/q}}_{L_p} < \infty.
	\end{equation*}
	\item For $s \in \mathbb{R}, 0<p\leq \infty$ and $0<q<\infty$ we define the (anisotropic) homogeneous decomposition space $\dot{M}_{p,q}^s(\tilde{h})$ as the set of all $f\in\mathcal{S}'\backslash\mathcal{P}$ satisfying
	\begin{equation*}
	\norm{f}_{\dot{M}_{p,q}^s(\tilde{h})} =
	\left( \sum_{j\in J} \norm{\tilde{h}(\xi_j)^s \phi_j(D)f}_{L_p}^q \right)^{1/q} < \infty,
	\end{equation*}
	with the usual modification if $q=\infty$.
\end{itemize}
\end{mydef}

It can be verified that are $\dot{F}_{p,q}^s(\tilde{h})$ and  $\dot{M}_{p,q}^s(\tilde{h})$ are (quasi-)Banach spaces that only (up to norm equivalence)  depend on $\tilde{h}$ and not the particular choice of BAPU, see \cite{AN19,BN08}. We  mention that it is possible to consider other reservoirs of distributions  than  $\mathcal{S}'\backslash\mathcal{P}$ to build the function spaces, see Voigtlaender \cite{2016arXiv160102201V} for further details. 

Next we construct a tight frame for the homogeneous T-L type spaces and the associated decomposition spaces. Further details can be found in \cite{AN19}. 

\subsection{Construction of frames}
Consider the modified BAPU $\{\phi_j\}_{j\in J}$ given by \eqref{eq:squarerootbapu} associated with the admissible covering $\mathcal{Q} := \{T_jQ\}_{j\in J} := \{T_jB_{\mathbf{a}}(0,\delta)\}_{j\in J}$ generated by $\{T_j \cdot = A_j\cdot + \xi_j\}_{j\in J}$. Suppose $K_a$ is a cube in $\mathbb{R}^d$ (aligned with the coordinate axes) with side-length $2a$ satisfying $B_{\mathbf{a}}(0,2\delta) \subseteq K_a$. Set 
\begin{equation}\label{eq:tj}
t_j := \tilde{h}(\xi_j).
\end{equation}
Then we define
\begin{equation*}
e_{j,n}(\xi) := (2a)^{-\tfrac{d}{2}}t_j^{-\tfrac{\nu}{2}} \chi_{K_a}(T_j^{-1}\xi)\e^{-i\tfrac{\pi}{a}n\cdot T_j^{-1}\xi}, \quad j \in J, n\in\mathbb{Z}^d,
\end{equation*}
and
\begin{equation*}
\hat{\eta}_{j,n} := \phi_j e_{j,n}, \quad j \in J, n\in \mathbb{Z}^d.
\end{equation*} 
It can easily   be shown that $\{\eta_{j,n}\}_{j\in J,n\in\Z^d}$ is a tight frame for $L_2(\R^d)$. Letting $\hat{\mu}_{j}(\xi) := \phi_{j}(T_j\xi)$, we obtain a representation of $\eta_{j,n}$ in direct space,
\begin{align}\label{eq:eta}
\eta_{j,n}(x) := (2a)^{-\frac{d}{2}} t_j^{\frac{\nu}{2}} \e^{ix\cdot \xi_j} \mu_j(A_j x-\tfrac{\pi}{a}n).
\end{align}
Since $\phi_j \in \mathcal{S}(\mathbb{R}^d)$ has compact support in $Q_j$, all of its partial derivatives are continuous and have compact support. Hence, for every $\gamma \in \mathbb{N}_0^d$ and some $C_\gamma > 0$ we have 
\begin{align}\label{eq:decayofpartialmu}
|\partial^{\gamma}_{\xi} \hat{\mu}_{j}(\xi)| = |\partial^{\gamma}_{\xi} \phi_j(T_j\xi)| \leq C_{\gamma}\chi_{B_{\mathbf{a}}(0,2\delta)}(\xi).
\end{align}
%Thus, for any $N\in \N$,
%\begin{align}\label{eq:mudecay}
%|\mu_{j}(x)| &\leq C \left| \frac{\sum_{|\alpha|\leq N} |x|^\alpha}{(1+|x|_{\ca})^N} \mu_j(x)\right| = C(1+|x|_{\ca})^{-N} \left| \sum_{|\alpha|\leq N} x^\alpha \mu_j(x)\right| \nonumber \\
%&\leq C(1+\abs{x}_{\ca})^{-N} \sum_{|\alpha|\leq N} \norm{\partial^\alpha_\xi \hat{\mu}_j}_{L_1} \nonumber \\
%&\leq C_N(1+\abs{x}_{\ca})^{-N},
%\end{align}
%with $C_N$ independent of $x\in \R^d$ and $j\in J$. 

We also need an estimate on $|\partial_x^\gamma \mu_j(x)|$. 
By the multinomial theorem it follows that, for any $N\in \mathbb{N}$ and $\beta \in \mathbb{N}_0^d$,
\[
|x|^N \sim (|x_1|+\cdots + |x_d|)^N = \sum_{|\beta|=N} \binom{N}{\beta} |\tilde{x}|^\beta,
\]
where $|\tilde{x}|^\beta = |x_1|^{\beta_1} \cdots |x_d|^{\beta_d}$. Thus, for any $\beta, \gamma \in \mathbb{N}_0^d$ we have
\begin{align*}
|(1+|x|_{\ca})^N \partial_x^\gamma \mu_j(x)| &\leq C \sum_{|\beta|\leq N} \binom{N}{\beta} |\tilde{x}^\beta \partial_x^\gamma \mu_j(x)| 
= C_{N} |\mathcal{F}^{-1} (\partial_\xi^\beta(\xi^\gamma \hat{\mu}_j(\xi)))| \\
&\leq C_N \int_{\mathbb{R}^d} |\partial_\xi^\beta (\xi^\gamma \hat{\mu}_j(\xi))|\d \xi \\
&\leq C_{N,\beta,\gamma}, 
\end{align*}
where the last inequality follows by \eqref{eq:decayofpartialmu}.
Rearranging terms yields
\begin{align}\label{eq:estimatepartialmu}
|\partial_x^\gamma \mu_j(x)| \leq C_{N,\beta,\gamma} (1+|x|_{\ca})^{-N}.
\end{align}

It turns out that $\{\eta_{j,n}\}_{j\in J,n\in\Z^d}$ constitutes a (universal) frame for $\dot{F}_{p,q}^s(\tilde{h})$ and $\dot{M}_{p,q}^s(\tilde{h})$. For a more precise statement of this fact we need to introduce the following associated sequence spaces. The following point sets will be useful for that.
\begin{equation}\label{eq:pointsets}
Q(j,n) = \left\{ y \in \R^d : A_jy + \frac{\pi}{a}n \in B_{\mathbf{a}}(0,1)\right\}
\end{equation}
It can be verified that there exists $n_0<\infty$ such that uniformly in $x$ and $j$ we have $\sum_{n\in\mathbb{Z}^d} \chi_{Q(j,n)} \leq n_0$. With this property in hand, we can define the sequence spaces associated with $\dot{F}_{p,q}^s(\tilde{h})$ and $\dot{M}_{p,q}^s(\tilde{h})$. 

\begin{mydef}
	Let $s\in\R, 0<p<\infty$ and $0<q\leq \infty$. We define the sequence space $\dot{f}_{p,q}^s(\tilde{h})$ as the set of all complex-valued sequences $\{c_{j,n}\}_{j\in J, n\in\mathbb{Z}^d} \subset \mathbb{C}$ satisfying
	\begin{equation*}
	\norm{c_{j,n}}_{\dot{f}_{p,q}^s(\tilde{h})} := \norm{\ \left( \sum_{j\in J} \sum_{n\in \Z^d} \left(t_j^{s+\frac{\nu}{2}} |c_{j,n}|\right)^q\chi_{Q(j,n)}\right)^{1/q}}_{L_p} < \infty 
	\end{equation*}
	Let $s\in \R, 0<p\leq \infty$ and $0<q<\infty$. We define the sequence space $\dot{m}_{p,q}^s(\tilde{h})$ as the set of all complex-valued sequences $\{c_{j,n}\}_{j\in J, n\in\mathbb{Z}^d} \subset \mathbb{C}$ satisfying
	\begin{equation*}
	\norm{c_{j,n}}_{\dot{m}_{p,q}^s(\tilde{h})} := \norm{ \left\{ t_j^{s + \frac{\nu}{2}-\frac{\nu}{p}} \left( \sum_{n\in\mathbb{Z}^d} \left|c_{j,n}\right|^p\right)^{1/p}\right\}_{j\in J}}_{\ell_q}
	\end{equation*}
	If $q=\infty$ or $p=\infty$ the $\ell_q$-norm or $\ell_p$-norm, respectively, is replaced by the $\ell_\infty$-norm.
\end{mydef}

Finally, we can verify that $\{\eta_{j,n}\}_{j\in J,n\in\Z^d}$ constitutes a frame for $\dot{F}_{p,q}^s(\tilde{h})$ and $\dot{M}_{p,q}^s(\tilde{h})$ in the following sense. 

\begin{myprop}\label{prop:frame}%
Assume that $s\in\R$, $0<p,q\le\infty$, $p<\infty$ for $\dot{F}^s_{p,q}(	\tilde{h})$, and $q<\infty$ for $\dot{M}^s_{p,q}(	\tilde{h})$. For any finite sequence $\{s_{k,n}\}_{k\in J,n\in \Z^d}\subset\mathbb{C}$, we have
\begin{equation*}
\Big\| \sum_{k\in J} \sum_{n\in \Z^d}s_{k,n}\eta_{k,n}\Big\|_{\dot{F}^s_{p,q}(\tilde{h})}\le C \|s_{k,n}\|_{\dot{f}^s_{p,q}(	\tilde{h})}.
\end{equation*}
Furthermore, $\{\eta_{k,n}\}_{k\in J,n\in\Z^d}$ is a frame for $\dot{F}^s_{p,q}(	\tilde{h})$,
\begin{equation*}
\|f\|_{\dot{F}^s_{p,q}(	\tilde{h})}\asymp\|\langle f, \eta_{k,n}\rangle\|_{\dot{f}^s_{p,q}(	\tilde{h})},\,\, f\in \dot{F}^s_{p,q}(	\tilde{h}).
\end{equation*}
Similar results hold for $\dot{M}^s_{p,q}(	\tilde{h})$ and $\dot{m}^s_{p,q}(\tilde{h})$.\\
\end{myprop}

The proof of Proposition \ref{prop:frame} in the $\dot{M}^s_{p,q}(\tilde{h})$-case can be found in  \cite{AN19}. The proof in the $\dot{F}^s_{p,q}(\tilde{h})$-case is similar using  the modifications outlined in \cite{BN08}.

\section{Almost diagonal matrices}\label{sec:almostdiagonal}
In this section we introduce and study a class of almost diagonal matrices for the sequence spaces $\dot{f}_{p,q}^s(\tilde{h})$ and $\dot{m}_{p,q}^s(\tilde{h})$ which corresponds to the function spaces $\dot{F}_{p,q}^s(\tilde{h})$ and $\dot{M}_{p,q}^s(\tilde{h})$, respectively. Our main contribution is to show that for any pairs of decomposition space frames with suitable localisation and smoothness properties, the corresponding  "change of frame" matrix will be almost diagonal. 

We say that a matrix $\mathbf{A} := \{a_{(k,m)(j,n)}\}_{k,j\in J,m,n\in \Z^d}$ belongs to the class $\text{ad}_{p,q}^s(\tilde{h})$ if its entries $| a_{(k,m)(j,n)}|$ decay at a certain rate apart from the diagonal. 
Based on the experience gained from earlier studies, see e.g.\ \cite{MR1070037,NR12,MR3303346,MR3981281}, we propose the following definition of almost diagonal matrices on  on $\dot{f}^s_{p,q}(h)$ and $\dot{m}^s_{p,q}(h)$.
\begin{mydef}\label{def:almostdiagonal}
Assume that $s\in \R$, $0<p,q\le \infty$, $p<\infty$ for $\dot{f}^s_{p,q}(h)$, and $q< \infty$ for $\dot{m}^s_{p,q}(h)$. Let $r:=\min(1,p,q)$. A matrix  $\mathbf{A} := \{a_{(k,m)(j,n)}\}_{k,j\in J,m,n\in \Z^d}$
is called almost diagonal on $\dot{f}^s_{p,q}(h)$ and $\dot{m}^s_{p,q}(h)$ if
there exist $C, \delta>0$ such that
\begin{align*}
|a_{(j,m)(k,n)}|
\le& C \bigg(\frac{t_k}{t_j}\bigg)^{s+\frac{\nu}{2}}
\min\bigg(\bigg(\frac{t_j}{t_k}\bigg)^{\frac{\nu}{r}+\frac{\delta}{2}},\bigg(\frac{t_k}{t_j}\bigg)^{\frac{\delta}{2}}\bigg) c_{jk}^\delta\\
&\phantom{C}\times
(1+\min(t_k,t_j)|x_{k,n}-x_{j,m}|_B)^{-\frac{\nu}{r}-\delta},
\end{align*}
where
\begin{equation*}
c_{jk}^\delta:=\min\bigg(\bigg(\frac{t_j}{t_k}\bigg)^{\frac{\nu}{r}+\delta},\bigg(\frac{t_k}{t_j}\bigg)^{\delta}\bigg)
(1+\max(t_k,t_j)^{-1}|\xi_k-\xi_j|_A)^{-\frac{\nu}{r}-\delta}
\end{equation*}	
%	
%	Assume $s \in \R, 0<p,q\leq\infty, p<\infty$ for $\dot{f}_{p,q}^s(\tilde{h})$ and $q<\infty$ for $\dot{m}_{p,q}^s(\tilde{h})$. Let $r:=\min(1,p,q)$. A matrix $\mathbf{A} := \{a_{(k,m)(j,n)}\}_{k,j\in J,m,n\in \Z^d}$ is called almost diagonal on $\dot{f}_{p,q}^s(\tilde{h})$ and $\dot{m}_{p,q}^s(\tilde{h})$ if there exists $C,\epsilon > 0$ such that
%\begin{align*}
%|a_{(k,m)(j,n)}| &\leq C \min\left( \frac{t_j}{t_k}, \frac{t_k}{t_j} \right)^{\nu/2 + \epsilon/2} (1+\max(t_j,t_k)^{-1}|\xi_j-\xi_k|_{\ca})^{-\nu/r - \epsilon} \\
%&\times (1+\min(t_j,t_k)|x_{j,n}-x_{k,m}|_{\ca})^{-\nu/r - \epsilon}, 
%\end{align*}
%	\begin{align*}
%	|a_{(k,m)(j,n)}| &\leq C \left(\frac{t_j}{t_k}\right)^{s+\frac{d}{2}} \min\left( \left(\frac{t_k}{t_j}\right)^{\frac{d}{r}+\frac{\epsilon}{2}}, \left(\frac{t_j}{t_k}\right)^{\frac{\epsilon}{2}} \right)b_{kj}^\epsilon \\
%	&\times (1+\min(t_j,t_k)|x_{j,n}-x_{k,m}|_{\ca})^{-\frac{d}{r}-\epsilon},
%	\end{align*}
%	where
%	\begin{align}\label{eq:bkj}
%	b_{kj}^\epsilon := \min\left(\left(\frac{t_k}{t_j}\right)^{\frac{d}{r}+\epsilon},\left(\frac{t_j}{t_k}\right)^\epsilon \right) (1+\max(t_j,t_k)^{-1}|\xi_j-\xi_k|_{\ca})^{-\frac{d}{r}-\epsilon}
%	\end{align}
	with $t_j$ defined in \eqref{eq:tj} and $x_{j,n}$ defined by
	\begin{equation}\label{eq:x}
	x_{j,n} = A_j^{-1} \frac{\pi}{a}n, \quad j\in J, n\in\Z^d.
	\end{equation}
	 We denote the set of almost diagonal matrices on $\dot{f}_{p,q}^s(\tilde{h})$ and $\dot{m}_{p,q}^s(\tilde{h})$ by $\text{ad}_{p,q}^s(\tilde{h})$. 
\end{mydef}

%\begin{myre}
%The asymmetry in $j$ and $k$ in Definition \ref{def:almostdiagonal} is only apparant as the condition applies to $|a_{(j,m)(k,n)}|=
%\end{myre}

There is an apparent similarity with the definition of almost diagonal matrices in the inhomogeneous setup considered in \cite{NR12}. However, it is important to notice that the sequence of dilation parameters $\{t_j\}_{j\in J}$ is not bounded away from zero in Definition \ref{def:almostdiagonal} due to the homogeneous setup unlike the case considered in  \cite{NR12}.

An important feature of almost diagonal matrices is stated in the following proposition showing that matrix composition is closed on the class of almost diagonal matrices. This will be useful when proving our main result of this section, Theorem \ref{thm:almostdiagonal}. Let us state the result, which is related to the corresponding result in the inhomogeneous case,  \cite[Proposition 3.4]{NR12}. However, since the dilation parameters $\{t_j\}$ are not bounded from below in the present homogeneous case, we have included a proof of Proposition \ref{prop:lukket} in the Appendix.

\begin{myprop}\label{prop:lukket}
The matrix product of almost diagonal matrices is almost diagonal. More precisely, we have $\text{ad}_{p,q}^s(\tilde{h}) \circ \text{ad}_{p,q}^s(\tilde{h}) \subseteq \text{ad}_{p,q}^s(\tilde{h})$.
\iffalse
Let $s\in\R, 0<r\leq 1$ and $\epsilon>0$. Define 
\begin{align*}
\omega_{(k,m),(j,n)}^{s,\epsilon} &:= \left(\frac{t_j}{t_k}\right)^{s+\frac{\nu}{2}} \min\left( \frac{t_j}{t_k}, \frac{t_k}{t_j} \right)^{\nu/2 + \epsilon/2} (1+\max(t_j,t_k)^{-1}|\xi_j-\xi_k|_{\ca})^{-\nu/r - \epsilon} \\
&\times (1+\min(t_j,t_k)|x_{j,n}-x_{k,m}|_{\ca})^{-\nu/r - \epsilon},
\end{align*}
with $t_j$ defined in \eqref{eq:tj} and $x_{j,n}$ in \eqref{eq:x}. Then 
\begin{align*}
\sum_{i\in J,l\in\Z^d} w^{s,\epsilon}_{(k,m)(i,l)} w^{s,\epsilon}_{(i,l)(j,n)} \leq C w_{(k,m)(j,n)}^{s,\epsilon/2}.
\end{align*}
\fi
\end{myprop}

\subsection{Almost diagonal change of frame matrices}
Our goal in this section is to estimate the entries in the  change of frame matrix associated with two suitably localised frames. We are following a gradual approach where we slowly remove restrictions to arrive at our main result, Theorem \ref{thm:almostdiagonal}.
 
 We first consider the frame $\{\eta_{j,n}\}_{j\in J,n\in\Z^d}$ defined in \eqref{eq:eta} together with another band-limited system $\{\psi_{k,m}\}_{j\in J,n\in\Z^d}$ satisfying  similar localisation conditions.	 

Our first goal is to prove Proposition \ref{prop:almostdiag}, which states that the "change of frame coefficient"-matrix $$\{\ip{\eta_{j,n},\psi_{k,m}}\}_{j,k,n,m}$$ is almost diagonal, where the functions $\{\psi_{k,m}\}_{k,m}$ are assumed to satisfy condition \eqref{eq:psidecay}-\eqref{eq:supppsi} below; in particular, they are decreasing functions in direct and frequency space and have compactly supported Fourier transform. Due to complications arising from the homogeneous structure of the spaces considered, the proof of  Proposition \ref{prop:almostdiag} will be somewhat technical and it relies on a number of Lemmas covering various spacial cases.  

It is interesting to note that for  {\em in}homogeneous spaces, the corresponding result is much more straightforward to prove, see \cite[Lemma 3.1]{NR12}. This is, to a large degree, due to the fact that the dilation parameters $\{t_j\}_j$ are bounded away from zero.  

 We begin our analysis with the following straightforward result that provides information on the decay properties of $\eta_{j,n}$ in both direct and frequency space.

\begin{mylemma}\label{prop:decayeta}
Choose $N,M$ such that $2N>\nu$ and $2M>\nu$. Let $t_j$ be given as in \eqref{eq:tj} and $A_j$ be given as in \eqref{eq:matrixA}. Then
	\begin{align}
	|\eta_{j,n}(x)| &\leq C_N t_j^{\frac{\nu}{2}}(1+t_j|x-x_{j,n}|_{\ca})^{-2N} \label{eq:etadecay} \\
	|\hat{\eta}_{j,n}(\xi)| &\leq C_Mt_j^{-\frac{\nu}{2}}(1+ t_j^{-1}|\xi-\xi_j|_{\ca})^{-2M} \label{eq:hatetadecay}
	\end{align}
	where $x_{j,n}$ is given by \eqref{eq:x}.
\end{mylemma}

\begin{proof}
	We begin with the estimate for $\eta_{j,n}(x)$. Using \eqref{eq:estimatepartialmu} with $\gamma=0$ and \eqref{eq:conditionaninorm} we have
	\begin{align*}
	|\eta_{j,n}(x)| &= \left| (2a)^{-\frac{d}{2}} t_j^{\frac{\nu}{2}} \e^{ix\cdot \xi_j} \mu_j(A_j x-\frac{\pi}{a}n) \right| 
	\leq (2a)^{-\frac{d}{2}} t_j^{\frac{\nu}{2}} \left|\mu_j(A_j x-\frac{\pi}{a}n)\right| \\
	&\leq C t_j^{\nu/2} (1+|A_j x-\tfrac{\pi}{a}n|_{\ca})^{-2N} \\
	&\leq C t_j^{\nu/2} (1+t_j|x-x_{j,n}|_{\ca})^{-2N},
	\end{align*}
	where $x_{j,n}$ is given in \eqref{eq:x}. 
	For the next estimate we use \eqref{eq:conditionaninorm} and \eqref{eq:decayofpartialmu} and get
	\begin{align*}
	|\hat{\eta}_{j,n}(\xi)| &= |\phi_j(\xi)e_{j,n}(\xi)| \leq Ct_j^{-\frac{\nu}{2}} \left| \phi_j(\xi)\e^{-i\frac{\pi}{a}n\cdot T_j^{-1}\xi}\right| 
	\leq C t_j^{-\frac{\nu}{2}} \left| \hat{\mu}_j(T_j^{-1}\xi) \right| \\
	&\leq C t_j^{-\frac{\nu}{2}} (1+|T_j^{-1}\xi|_{\ca})^{-2M}
	= C t_j^{-\frac{\nu}{2}} (1+|A_j^{-1}(\xi-\xi_j)|_{\ca})^{-2M} \\
	&= C t_j^{-\frac{\nu}{2}} (1+ t_j^{-1}|\xi-\xi_j|_{\ca})^{-2M}.
	\end{align*}
\end{proof}
We now turn to the actual estimation of the "change of frame"-matrix in various settings, leading to our main result, Theorem \ref{thm:almostdiagonal}. Our first main result will be Proposition \ref{prop:almostdiag} that considers systems with the following band-limited structure.  Let $\{\psi_{k,m}\}_{k\in J,m\in\Z^d} \subset L_2(\R^d)$ be an arbitrary sequence of functions with similar decay properties as our original frame $\{\eta_{j,n}\}_{j\in J,n\in\Z^d}$, and assume  that the functions $\{\psi_{k,m}\}_{k,m}$ are band-limited and compatible with the decomposition of the frequency space. That is,
\begin{align}
|\psi_{k,m}(x)| &\leq Ct_k^{\frac{\nu}{2}}(1+ t_k|x-x_{k,m}|_{\ca})^{-2N'} \label{eq:psidecay} \\
|\hat{\psi}_{k,m}(\xi)| &\leq Ct_k^{-\frac{\nu}{2}}(1+ t_k^{-1}|\xi-\xi_k|_{\ca})^{-2M'}, \quad \text{and} \label{eq:hatpsidecay} \\
\text{supp}(\hat{\psi}_{k,m}) &\subseteq Q_k, \label{eq:supppsi}
\end{align}
where the constant $C$ is independent of $k$ and $m$. 
With these assumptions we focus on estimating  $|\ip{\eta_{j,n},\psi_{k,m}}|$. Let us first make the observation that the associated functions $$\upsilon_{k,m}:=(2a)^{d/2}t_k^{-\nu/2}e^{-iA_k^{-1}(\cdot+\frac{\pi}{a}m )\cdot\xi_k}\psi_{k,m}\big(A_k^{-1}\big(\cdot+\frac{\pi}{a}m\big)\big)$$
satisfy, using  \eqref{eq:psidecay},
$$|\upsilon_{k,m}(x)| \leq C(1+ |x|_{\ca})^{-2N'}, $$
with $C$ independent of $k$ and $m$, while
\begin{equation}\label{eq:structure}
\psi_{k,m}(x)=(2a)^{-\frac{d}{2}} t_k^{\frac{\nu}{2}} \e^{ix\cdot \xi_k} \upsilon_{k,m}(A_k x-\tfrac{\pi}{a}m).
\end{equation} 

 To prove Proposition \ref{prop:almostdiag}, we need to consider a number of lemmas.   

\begin{mylemma}\label{lemma:1}
Choose $N>\nu$ and suppose $\{\eta_{j,n}\}_{j\in J,n\in\Z^d}$ satisfies \eqref{eq:etadecay}, and $\{\psi_{k,m}\}_{k\in J,m\in\Z^d}$ satisfies \eqref{eq:psidecay} with $N'\geq N$. Then
	\begin{equation}\label{eq:estimate1}
	\left|\ip{\eta_{j,n},\psi_{k,m}} \right| \leq C \min\left(\frac{t_j}{t_k},\frac{t_k}{t_j} \right)^{\frac{\nu}{2}} (1+\min(t_j,t_k)|x_{j,n}-x_{k,m}|_{\ca})^{-N},
	\end{equation}
	where $t_j$ is defined in \eqref{eq:tj} and $x_{j,n}$ in \eqref{eq:x}.
\end{mylemma}

\begin{proof}
	Without loss of generality assume that $t_j \leq t_k$. We consider two cases. 

Case 1: Suppose $t_j|x_{j,n}-x_{k,m}|_{\ca} \leq 1$. Since $N>\nu$, it follows that
		\begin{equation}\label{eq:tjestimates}
		\frac{t_j^{\frac{\nu}{2}}}{(1+ t_j|x-x_{j,n}|_{\ca})^N} \leq t_j^{\frac{\nu}{2}} \leq \frac{2^Nt_j^{\frac{\nu}{2}}}{(1+ t_j|x_{j,n}-x_{k,m}|_{\ca})^N}.
		\end{equation}
		Using \eqref{eq:tjestimates}, the decay properties of $\eta_{j,n}(x)$, and $\psi_{k,m}(x)$ and by a change of variable, we obtain
		\begin{align}\label{eq:case1estimates}
		\left|\ip{\eta_{j,n},\psi_{k,m}}\right| &\leq \int_{\mathbb{R}^d} \frac{Ct_j^{\frac{\nu}{2}}}{(1+ t_j|x-x_{j,n}|_{\ca})^N} \frac{t_k^{\frac{\nu}{2}}}{(1+ t_k|x-x_{k,m}|_{\ca})^N} \d x \nonumber \\
		&\leq \frac{Ct_j^{\frac{\nu}{2}}}{(1+t_j|x_{j,n}-x_{k,m}|_{\ca})^N} \int_{\mathbb{R}^d} \frac{t_k^{\frac{\nu}{2}}}{(1+ t_k|x-x_{k,m}|_{\ca})^N} \d x \nonumber \\
		&= \frac{Ct_j^{\frac{\nu}{2}}}{(1+ t_j|x_{j,n}-x_{k,m}|_{\ca})^N} \int_{\mathbb{R}^d} \frac{t_k^{-\frac{\nu}{2}}}{(1+|u|_{\ca})^N} \d u \nonumber \\
		&\leq C\left(\frac{t_j}{t_k} \right)^{\frac{\nu}{2}} (1+ t_j|x_{j,n}-x_{k,m}|_{\ca})^{-N}.
		\end{align}
		
Case 2:
		Now suppose $t_j|x_{j,n}-x_{k,m}|_{\ca} > 1$, and assume first that $|x-x_{j,n}|_{\ca} \geq \frac{1}{2K} |x_{j,n}-x_{k,m}|_{\ca}$, with $K$ given in Definition \ref{def:aninorm}. Similar to above we then get \eqref{eq:tjestimates} which leads to \eqref{eq:case1estimates}. 
		Now, assume $|x-x_{j,n}|_{\ca} < \frac{1}{2K} |x_{j,n}-x_{k,m}|_{\ca}$. Then it follows that $\frac{1}{2K} |x_{j,n}-x_{k,m}|_{\ca} < |x-x_{k,m}|_{\ca}$.
%		\begin{align*}
%		|x_{j,n}-x_{k,m}|_{\ca} &\leq K(|x_{j,n}-x|_{\ca} + |x-x_{k,m}|_{\ca}) &\Leftrightarrow \\ \frac{1}{K}|x_{j,n}-x_{k,m}|_{\ca} &\leq \frac{1}{2K}|x_{j,n}-x_{k,m}|_{\ca} + |x-x_{k,m}|_{\ca} &\Leftrightarrow \\
%		\frac{1}{2K} |x_{j,n}-x_{k,m}|_{\ca} &< |x-x_{k,m}|_{\ca}.
%		\end{align*}	
Thus we have
		\begin{align*}
		\frac{1}{(1+t_k|x-x_{k,m}|_{\ca})^N} &\leq \frac{C}{(1+t_k|x_{j,n}-x_{k,m}|_{\ca})^N} = \frac{C(t_j/t_k)^N}{\left( (t_j/t_k)+t_j|x_{j,n}-x_{k,m}|_{\ca} \right)^N} \\
		&\leq \frac{C(t_j/t_k)^N}{(t_j|x_{j,n}-x_{k,m}|_{\ca})^N} 
		\leq  \frac{2^NC(t_j/t_k)^N}{(2t_j|x_{j,n}-x_{k,m}|_{\ca})^N} \\
		&\leq \frac{C(t_j/t_k)^N}{\left( 1+t_j|x_{j,n}-x_{k,m}|_{\ca} \right)^N}.
		\end{align*}
		Since, by assumption, $t_j/t_k\leq 1$, we now use that $N>\nu$ to obtain
		\begin{align*}
		\left|\ip{\eta_{j,n},\psi_{k,m}}\right| &\leq \int_{\mathbb{R}^d} \frac{Ct_j^{\frac{\nu}{2}}}{(1+ t_j|x-x_{j,n}|_{\ca})^N} \frac{t_k^{\frac{\nu}{2}}}{(1+ t_k|x-x_{k,m}|_{\ca})^N} \d x \\
		&\leq \int_{\mathbb{R}^d} \frac{Ct_j^{\frac{\nu}{2}}}{(1+ t_j|x-x_{j,n}|_{\ca})^N} \frac{C t_k^{\frac{\nu}{2}}(t_j/t_k)^N}{(1+t_j|x_{j,n}-x_{k,m}|_{\ca})^N} \d x \\
		&\leq \int_{\mathbb{R}^d} \frac{C t_j^{\frac{\nu}{2}}}{(1+ t_j|x-x_{j,n}|_{\ca})^N} \frac{t_k^{\frac{\nu}{2}} (t_j/t_k)^{\frac{\nu}{2}}(t_j/t_k)^{\frac{\nu}{2}}}{(1+t_j|x_{j,n}-x_{k,m}|_{\ca})^N} \d x \\
		&= \frac{C(t_j/t_k)^{\frac{\nu}{2}}}{(1+t_j|x_{j,n}-x_{k,m}|_{\ca})^N} \int_{\mathbb{R}^d} \frac{t_j^\nu}{(1+ t_j|x-x_{j,n}|_{\ca})^N} \d x \\
		&\leq C \left(\frac{t_j}{t_k}\right)^{\frac{\nu}{2}} (1+t_j|x_{j,n}-x_{k,m}|_{\ca})^{-N}.
		\end{align*}

	Thus the required estimate follows.
\end{proof}

We will use Lemma \ref{lemma:1} to prove Proposition \ref{prop:almostdiag}. However, we also need a stronger estimate in the case where $\min(t_j,t_k) <1$. This will be addressed in Lemma \ref{lemma:bandlimited}. The proof of Lemma \ref{lemma:bandlimited} will rely on the following lemma.

\begin{mylemma}\label{lemma:decay}
	Let $L\in \mathbb{N}$ and let $\alpha_1, \alpha_2$ be given by \eqref{eq:aniless1}. Choose $N>\nu$ and $R> 2N+ L/\alpha_1$ and assume $t_j\leq t_k$. Suppose the functions $f_j\in \mathcal{C}^{L}(\mathbb{R}^d)$ and $g_k\in L_1(\mathbb{R}^d)$ satisfy
	\begin{align}
	|\partial_x^\gamma f_j(x)| &\leq C_1^\gamma t_j^{\nu/2} t_j^{\alpha_1 L} (1+t_j|x-x_j|_{\ca})^{-N}, \quad |\gamma|=L \label{eq:partialf}. \\
	|g_k(x)| &\leq C_2 t_k^{\nu/2}(1+t_k|x-x_k|_{\ca})^{-R}. \label{eq:gdecay} \\
	\int_{\R^d} x^{\beta} g_k(x) \d x &= 0, \quad |\beta| \leq L-1. \label{eq:vanising}
	\end{align}
	Then there exists a constant $C>0$, independent of $f_j, g_k, t_j, t_k, x_j$ and $x_k$, such that
	\begin{align*}
	|\ip{f_j,g_k}| \leq C\tilde{C} \frac{t_j^{\nu/2 + \alpha_1 L}}{t_k^{\nu/2}\min(t_k^{L/\alpha_1}, t_k^{L/\alpha_2})} (1+t_j|x_j-x_k|_{\ca})^{-N},
	\end{align*}
	where $\tilde{C} = \left(\sum_{|\gamma|=L} C_1^\gamma\right) C_2$.
	
\end{mylemma}

\begin{proof}
	By the vanishing moment condition \eqref{eq:vanising}, we have
	\begin{align*}
	\left|\int_{\mathbb{R}^d} f_j(x) g_k(x) \d x \right| \leq \int_{\mathbb{R}^d} \left|f_j(x) - \sum_{|\gamma|\leq L-1} \frac{\partial^\gamma f_j(x_k)}{\gamma!} (x-x_k)^\gamma\right| |g_k(x)| \d x
	\end{align*}
	Using the Taylor Remainder Theorem, and \eqref{eq:partialf} with $|\gamma|=L$ together with \eqref{eq:gdecay}, we get
	\begin{align*}
	& \left|\int_{\mathbb{R}^d} f_j(x) g_k(x) \d x \right| \nonumber \\
	&\leq C \int_{\mathbb{R}^d} |x-x_k|^L  |\partial^\gamma_y f_j(y)| |g_k(x)|\d x \nonumber \\
	&\leq C \int_{\mathbb{R}^d} \max\{|x-x_k|_{\ca}^{1/\alpha_1}, |x-x_k|_{\ca}^{1/\alpha_2}\}^L |\partial^{\gamma}_y f_j(y)| |g_k(x)|\d x \nonumber \\
	&\leq C \tilde{C} \int_{\mathbb{R}^d} 
	\frac{\max\{|x-x_k|_{\ca}^{L/\alpha_1}, |x-x_k|_{\ca}^{L/\alpha_2}\} \; t_j^{\nu/2} t_j^{\alpha_1 L}}{(1+t_j|y-x_j|_{\vec{a}})^{N}}
	\frac{t_k^{\nu/2}}{(1+t_k|x-x_k|_{\ca})^{R}} \; \d x,
	\end{align*}
	for some $y$ on the line segment joining $x_k$ and $x$. Using $t_j\leq t_k$, and the quasi-triangle inequality, we have
	\begin{align}\label{eq:estimatemedy}
	\frac{1}{K}\frac{1}{1+t_j|y-x_j|_{\ca}} \leq K 
	\frac{1+t_k|x-x_k|_{\ca}}{1+t_j|x_j-x_k|_{\ca}}.
	\end{align}  
	Inserting this estimate in the last integral, and by a change of variable, we obtain
	\begin{align*}
	&\left|\int_{\mathbb{R}^d} f_j(x) g_k(x) \d x \right| \nonumber \\
	&\leq C \tilde{C} \int_{\mathbb{R}^d} \frac{\max\{|x-x_k|_{\ca}^{L/\alpha_1}, |x-x_k|_{\ca}^{L/\alpha_2}\} \; t_j^{\nu/2+\alpha_1 L}}{(1+t_j|x_j-x_k|_{\ca})^N} \frac{t_k^{\nu/2}}{(1+t_k|x-x_k|_{\ca})^{R-N}} \; \d x \\
	&\leq C\tilde{C} \frac{t_j^{\nu/2+\alpha_1 L}}{(1+t_j|x_j-x_k|_{\ca})^N} \int_{\mathbb{R}^d} \frac{\max\{|x-x_k|_{\ca}^{L/\alpha_1}, |x-x_k|_{\ca}^{L/\alpha_2}\}
		t_k^{\nu/2}}{(1+t_k|x-x_k|_{\ca})^{R-N}} \; \d x \\
	&\leq C\tilde{C} \frac{t_j^{\nu/2+\alpha_1 L}}{(1+t_j|x_j-x_k|_{\ca})^N} \int_{\mathbb{R}^d} \frac{\max\{(t_k^{-1}|u|_{\ca})^{L/\alpha_1},(t_k^{-1}|u|_{\ca})^{L/\alpha_2}\} \; t_k^{-\nu/2}}{(1+|u|_{\ca})^{R-N}} \d u \\
	&\leq C\tilde{C} \frac{t_j^{\nu/2+\alpha_1L}}{t_k^{\nu/2}\min(t_k^{L/\alpha_1}, t_k^{L/\alpha_2})} (1+t_j|x_j-x_k|_{\ca})^{-N},
	\end{align*}
	where the last inequality follows since $R > 2N + L/\alpha_1$.
\end{proof}

We are now ready to prove the following.
\begin{mylemma}\label{lemma:bandlimited}
Let $L\in \mathbb{N}$ and choose $M,N>\nu$. Let $\{\eta_{j,n}\}_{j\in J,n\in\Z^d}$ be the frame defined in \eqref{eq:eta} satisfying \eqref{eq:etadecay}, and let $\{\psi_{k,m}\}_{k\in J,m\in\Z^d}$ satisfy \eqref{eq:psidecay}, \eqref{eq:hatpsidecay}, and \eqref{eq:supppsi} with $N'>2N+L/\alpha_1$ and $M'\geq M$. Assume $\min(t_j,t_k) < 1$. Then there exists a constant $C_L>0$ such that
	\begin{align*}
	|\ip{\eta_{j,n},\psi_{k,m}}| \leq C_L \min\left(\frac{t_j}{t_k}, \frac{t_k}{t_j}\right)^{\nu/2+L/\alpha_1} (1+\min(t_j,t_k)|x_{j,n}-x_{k,m}|_{\ca})^{-2N}.
	\end{align*} 
\end{mylemma}

\begin{proof}
Without loss of generality assume that $t_j\leq t_k$. 
	With $\eta_{j,n}(x)$ given in \eqref{eq:eta} and $\psi_{k,m}(x)$ satisfying \eqref{eq:structure}, we have
	\begin{align*}
	|\ip{\eta_{j,n},\psi_{k,m}}| \leq C \int_{\mathbb{R}^d} |t_j^{\nu/2} e^{ix\cdot\xi_j} \mu_j(A_j x-\tfrac{\pi}{a}n) \; t_k^{\nu/2}  e^{-ix\cdot\xi_k} \upsilon_{k,m}(A_k x-\tfrac{\pi}{a}m)| \d x.
	\end{align*} 
	Let $f_{j,n}(x) = t_j^{\nu/2}\mu_j(A_j x-\tfrac{\pi}{a}n)$ and $g_{k,m}(x)=t_k^{\nu/2} e^{ix\cdot(\xi_j-\xi_k)} \upsilon_{k,m}(A_k x-\tfrac{\pi}{a}m)$. We first consider $|\partial_x^\gamma f_{j,n}(x)|$. Applying the chain rule, together with the estimate \eqref{eq:estimatepartialmu}, we find that
	\begin{align}\label{eq:afledtefjn}
	|\partial_x^\gamma f_{j,n}(x)| &\leq |t_j^{\nu/2} \partial_x^\gamma \mu_j(A_j x-\tfrac{\pi}{a}n)| \leq |t_j^{\nu/2} t_j^{\ca\cdot \gamma}  (\partial_x^\gamma\mu_j)(A_j x-\tfrac{\pi}{a}n)| \nonumber \\
	&\leq C t_j^{\nu/2} t_j^{\alpha_1 L} (1+|A_j x-\tfrac{\pi}{a}n|_{\ca})^{-2N} \nonumber \\
	&\leq C t_j^{\nu/2} t_j^{\alpha_1L} (1+t_j|x-x_{j,n}|_{\ca})^{-2N}, \quad |\gamma|=L,
	\end{align}
	with $x_{j,n}$ defined in \eqref{eq:x}.  
	Now, by definition of $g_{k,m}(x)$, and since $\hat{\psi}_{k,m}(\xi)$ has support in $Q_k$, it follows that $\text{supp}(\hat{g}) \subseteq Q_k - \xi_j $. We consider two cases.
	
	Case 1: $\{0\} \in Q_k-\xi_j$. Then $\xi_j\in Q_k$ and $Q_j \cap Q_k \ne \emptyset$. By the moderation of $\tilde{h}$ we have $t_j\asymp t_k$, thus $\left(\frac{t_j}{t_k} \right) \asymp 1$. Using this together with the estimate \eqref{eq:estimate1} from Lemma \ref{lemma:1}, we multiply by a factor of $1$, and use that $\big(\frac{t_k}{t_j}\big)^M \leq C^M$ for any $M\in\mathbb{N}$ for some $C:=C(M)>0$, to obtain
		\begin{equation*}
		\left|\ip{\eta_{j,n},\psi_{k,m}} \right| \leq C \left(\frac{t_j}{t_k}\right)^{\frac{\nu}{2}+M} (1+t_j|x_{j,n}-x_{k,m}|_{\ca})^{-2N}.
		\end{equation*} 
		Choosing $M=L/\alpha_1$ gives the required estimate.
		
		Case 2: $\{0\} \not\in Q_k-\xi_j$. Here $Q_j \cap Q_k = \emptyset$. Then $g_{k,m}(x)$ satisfies the vanishing moment condition \eqref{eq:vanising}. Moreover, by the decay properties of $g_{k,m}(x)$ and \eqref{eq:afledtefjn} we may use Lemma \ref{lemma:decay} to conclude that
		\begin{align*}
		|\ip{\eta_{j,n},\psi_{k,m}}| \leq C \left(\frac{t_j}{t_k}\right)^{\nu/2+L/\alpha_1} (1+t_j|x_{j,n}-x_{k,m}|_{\ca})^{-2N}.
		\end{align*}
		
\end{proof}

In order to prove of Proposition \ref{prop:almostdiag}, 
we need to add one further restriction on the hybrid regulation function $\tilde{h}$ from Definition \ref{def:hfunction}. From now on we assume that the function $h_2(\xi)$ in \eqref{eq:hybrid} satisfies the following:
\begin{equation}\label{eq:ekstrabetingelse}
\begin{cases}
\text{There exists} \; \beta, R_1, \rho_1 > 0 \; \text{such that} \; h_2^{1+\beta} \; \text{is d-moderate and} \\
|\xi-\zeta|_{\ca} \leq ah_2(\xi) \; \text{for} \; a\geq \rho_1 \; \text{implies} \; h_2(\zeta) \leq R_1 a h_2(\xi).
\end{cases}
\end{equation}

\begin{myre}
The added restriction on $h_2$ is not very prohibitive as we can generate a multitude of  such functions by using $s:\R_+ \rightarrow \R_+$ satisfying $s(2b) \leq Cs(b), b \in \R_+$, and
$$(1+b)^\gamma \leq s(b) \leq (1+b)^{\frac{1}{1+\beta}}$$
for some $\beta,\gamma  > 0$. We assign $h_2 = s(| \cdot |_\mathbf{a})$ and use that $s$ is weakly sub-additive to verify \eqref{eq:ekstrabetingelse}. For instance, any regulation function from Example \ref{eks:hrf} will work provided $\alpha<1$.
\end{myre}

We are now ready to prove the following result.

\begin{myprop}\label{prop:almostdiag}
Let $L >0$  and choose $N,M$ such that $2N>\nu$ and $2M >\nu$. 
Let $\{\eta_{j,n}\}_{j\in J,n\in\Z^d}$ be the frame defined in \eqref{eq:eta}  and suppose $\{\psi_{k,m}\}_{k\in J,m\in Z^d}$ satisfies  \eqref{eq:psidecay}, \eqref{eq:hatpsidecay}, and \eqref{eq:supppsi}  with $N'>2N+L/\alpha_1$ and $M'> M+\frac{L}{\beta\alpha_1}$. Then there exists a constant $C:=C({L})>0$ such that
	\begin{align*}
	\left|\ip{\eta_{j,n},\psi_{k,m}} \right| &\leq C\min\left(\frac{t_j}{t_k},\frac{t_k}{t_j} \right)^{\frac{\nu}{2}+\frac{L}{\alpha_1}} (1+\max(t_j,t_k)^{-1}|\xi_j-\xi_k|_{\ca})^{-M} \nonumber \\
	&\times (1+\min(t_j,t_k)|x_{j,n}-x_{k,m}|_{\ca})^{-N}.
	\end{align*}
\end{myprop}

\begin{proof}
	We split the proof into three different cases.
	\begin{itemize}
		\item[Case 1:] Suppose $t_j\leq t_k$ and $t_j <1$. Using Lemma \ref{lemma:bandlimited} gives
		\begin{align}\label{case1:etapsi}
		|\ip{\eta_{j,n},\psi_{k,m}}| \leq C \left(\frac{t_j}{t_k}\right)^{\nu/2 + L/\alpha_1} (1+t_j|x_{j,n}-x_{k,m}|_{\ca})^{-2N},
		\end{align}
where $L=|\gamma|$ and $\alpha_1 > 1$ (as in Lemma \ref{lemma:bandlimited}). Moreover, using Lemma \ref{lemma:1} for $\langle\hat{\eta}_{j,n},\hat{\psi}_{k,m}\rangle$ gives
		\begin{align}\label{case1:hatetapsi}
		|\langle\hat{\eta}_{j,n},\hat{\psi}_{k,m}\rangle| \leq C \left(\frac{t_j}{t_k}\right)^{\nu/2} (1+t_k^{-1}|\xi_j-\xi_k|_{\ca})^{-2M}.
		\end{align}
		Now, combing the estimates \eqref{case1:etapsi} and \eqref{case1:hatetapsi}, and using that
		\begin{align}\label{eq:sqrt}
		|\ip{\eta_{j,n},\psi_{k,m}}| = |\ip{\eta_{j,n},\psi_{k,m}}|^{1/2} |\langle \hat{\eta}_{j,n},\hat{\psi}_{k,m}\rangle|^{1/2},
		\end{align}
		we obtain
		\begin{align*}
		|\ip{\eta_{j,n},\psi_{k,m}}| &\leq C \left(\frac{t_j}{t_k} \right)^{\nu/2+\tilde{L}} (1+t_j|x_{j,n}-x_{k,m}|_{\ca})^{-N} (1+t_k^{-1}|\xi_j-\xi_k|_{\ca})^{-M}.
		\end{align*}
		
		\item[Case 2:] Suppose $t_k\leq t_j$ and $t_k < 1$. By using similar arguments as in case 1 we obtain the required estimate.
		
		\item[Case 3:] Finally, suppose $t_j\leq t_k$ and $t_j\geq 1$. We first consider the case $|\xi_j-\xi_k|_{\ca} \leq \rho_0 t_k^{1+\beta}$. Since $t_j \geq 1$, the hybrid regulation function $\tilde{h}^{1+\beta}$ is moderate by definition of $\tilde{h}$, \eqref{eq:hybrid}, and assumption \eqref{eq:ekstrabetingelse}. Thus we have
		\begin{align}\label{case3_1}
		\frac{1}{1+t_k^{-1}|\xi_j-\xi_k|_{\ca}}  \leq 1 \leq R_0^{\tfrac{\beta}{1+\beta}} \left(\frac{t_j}{t_k} \right)^\beta.
		\end{align}
		Now consider the case $|\xi_j-\xi_k|_{\ca} > \rho_0t_k^{1+\beta}$. Since $t_j \geq 1$ we get
		\begin{align}\label{case3_2}
		\frac{1}{1+t_k^{-1}|\xi_j-\xi_k|_{\ca}} \leq \frac{1}{1+t_k^{-1}\rho_0 t_k^{1+\beta}} \leq \frac{1}{\rho_0t_k^\beta}\leq \frac{1}{\rho_0} \left( \frac{t_j}{t_k}\right)^\beta.  
		\end{align}
		Using Lemma \ref{lemma:1} for $\langle \hat{\eta}_{j,n},\hat{\psi}_{k,m}\rangle$ together with the estimates \eqref{case3_1} and \eqref{case3_2}, we obtain
		\begin{align}\label{case3_3}
		|\langle \hat{\eta}_{j,n},\hat{\psi}_{k,m}\rangle| &\leq C \left(\frac{t_j}{t_k} \right)^{\nu/2} (1+t_k^{-1}|\xi_j-\xi_k|_{\ca})^{-2M-2\frac{{L}}{\beta\alpha_1}} \nonumber \\
		&\leq C\left(\frac{t_j}{t_k} \right)^{\nu/2+2{L}/\alpha_1}(1+t_k^{-1}|\xi_j-\xi_k|_{\ca})^{-2M}.
		\end{align}
		Combining \eqref{case3_3} with \eqref{eq:estimate1} from Lemma \ref{lemma:1} and using \eqref{eq:sqrt} we obtain the required estimate. 
		\end{itemize}	
\end{proof}

In Proposition \ref{prop:almostdiag} we assumed that the functions in $\{\psi_{k,m}\}_{k\in J,m\in\Z^d}$ have compact support in the frequency domain $\mathbb{R}^d\backslash\{0\}$. In the following, we omit this assumption and consider a system $\{\psi_{k,m}\}_{k,m}$ satisfying only condition \eqref{eq:psidecay},  \eqref{eq:hatpsidecay}  together with our original frame $\{\eta_{j,n}\}_{j\in J,n\in \Z^d}$ defined in \eqref{eq:eta}. 
We first notice that the proof of Proposition \ref{prop:almostdiag} only used the assumption about compact support in frequency for $\{\psi_{k,m}\}_{k,m}$ in Cases 1 and 2, that is when $\min(t_j,t_k) < 1$. Thus, in the proof of following lemma, we only consider these cases. 

\begin{mylemma}\label{lemma:enikkebandlimited}
Let $L\in\N$ and choose $N>\nu+L/\alpha_1$. Let $\{\eta_{j,n}\}_{j\in J,n\in\Z^d}$ be the frame defined in \eqref{eq:eta}  and let $\{\psi_{k,m}\}_{k\in J,m\in Z^d}$ satisfy \eqref{eq:psidecay} and \eqref{eq:hatpsidecay} with $N'>2N+L/\alpha_1$ and $M'>M+\frac{L}{\beta\alpha_1}$. Assume $\min(t_j, t_k)<1$. Then there exists a constant $C:=C(L)>0$ such that
\begin{align*}
\left|\ip{\eta_{j,n},\psi_{k,m}} \right| &\leq C \min\left(\frac{t_j}{t_k},\frac{t_k}{t_j}\right)^{\alpha_1 L - L/\alpha_1} (1+\min(t_j,t_k)|x_{j,n}-x_{k,m}|_{\ca})^{-N}.
\end{align*}
\end{mylemma}

\begin{proof}
Without loss of generality assume that $t_j \leq t_k$. We start by considering the case $t_k\geq 1$. We have, by \eqref{eq:structure},
\begin{align}\label{eq:startproof}
\left|\ip{\eta_{j,n},\psi_{k,m}}\right| \leq C \int_{\mathbb{R}^d} |t_j^{\nu/2} e^{ix\cdot\xi_j} \mu_j(A_j x-\tfrac{\pi}{a}n) \; t_k^{\nu/2}  e^{-ix\cdot\xi_k} \upsilon_{k,m}(A_k x-\tfrac{\pi}{a}m)| \d x.
\end{align} 
By a change of variable, letting $u=A_k x$, we obtain
\begin{align}\label{eq:novan}
\left|\ip{\eta_{j,n},\psi_{k,m}}\right| &\leq C \int_{\mathbb{R}^d} |t_j^{\nu/2} e^{iA_k^{-1}u(\xi_j-\xi_k)} \mu_j(A_j A_k^{-1}u-\tfrac{\pi}{a}n) \; t_k^{\nu/2} \upsilon_{k,m}(u-\tfrac{\pi}{a}m)| t_k^{-\nu} \d u \nonumber \\
&\leq C \left(\frac{t_j}{t_k}\right)^{\nu/2} \int_{\mathbb{R}^d} | e^{iA_k^{-1}u(\xi_j-\xi_k)} \mu_j(A_j A_k^{-1}u-\tfrac{\pi}{a}n) \upsilon_{k,m}(u-\tfrac{\pi}{a}m) | \d u.
\end{align} 
Our wish is to use Lemma \ref{lemma:decay}. However, we first need to clarify that all the assumptions are satisfied. 
We begin by considering $|\partial_u^\gamma \mu_j (A_j A_k^{-1}u-\tfrac{\pi}{a}n)|$. Using the chain rule, and the estimate \eqref{eq:estimatepartialmu}, we find that
\begin{align}\label{eq:1}
|\partial_u^\gamma \mu_j (A_j A_k^{-1}u-\tfrac{\pi}{a}n)| &= \left| \left(\frac{t_j}{t_k} \right)^{\ca\cdot \gamma} (\partial_u^\gamma \mu_j)(A_j A_k^{-1}u - \frac{\pi}{a}n)\right| \nonumber \\
&\leq C \left(\frac{t_j}{t_k} \right)^{\alpha_1 L} \left(1+ |A_j A_k^{-1}u-\frac{\pi}{a}n|_{\ca}\right)^{-2N} \nonumber \\
&\leq C \left(\frac{t_j}{t_k} \right)^{\alpha_1 L} \left( 1 + \frac{t_j}{t_k} \left|u-u_{j,k,n}\right|_{\ca}\right)^{-2N},
\end{align}
where $|\gamma|=L$ and $u_{j,k,n} = A_j^{-1}A_k \frac{\pi}{a}n$.
Using that $\eta_{j,n}$ has compact support in frequency, $\hat{\eta}_{j,n} \subseteq Q_j$, we define a set $E$ as follows:
\begin{align*}
E = \text{supp}[ \mathcal{F}\{{\e^{iA_k^{-1}u(\xi_j-\xi_k)}\mu_j(A_jA_k^{-1}\cdot-\frac{\pi}{a}n)}\} ],
\end{align*} 
where $\mathcal{F}$ denotes the Fourier transform. Thus $E \subseteq Q_j - \xi_k$ and we distinguish two cases.

Case 1: $\{0\} \in E$. In this case $\xi_k \in Q_j$ and $Q_j \cap Q_k \ne \emptyset$. Using similar arguments as in Case 1 in the proof of Lemma \ref{lemma:bandlimited} we obtain the required estimate.

Case 2: $\{0\} \not\in E$. Here $Q_j\cap Q_k = \emptyset$. Now choose a smooth bump function $\hat{\rho}(\xi)$ that is equal to $1$ when $\xi \in E$ and equal to zero when $\xi$ is outside of $Q_j$. Then we may rewrite \eqref{eq:novan}, using \eqref{eq:structure}, as
\begin{align}\label{eq:integralmedfoldning}
\left|\ip{\eta_{j,n},\psi_{k,m}}\right| \leq C \left(\frac{t_j}{t_k} \right)^{\nu/2} \int_{\mathbb{R}^d} |%\e^{iA_k^{-1}u(\xi_j-\xi_k)}
\mu_j(A_j A_k^{-1}u-\frac{\pi}{a}n) \rho * \upsilon_{k,m}(u-\frac{\pi}{a}m) | \d u.
\end{align}
The function $\rho * \upsilon_{k,m}(u-\frac{\pi}{a}m)$ has compact support in the frequency domain $\hat{\rho} \hat{\psi}_k(u-\frac{\pi}{a}m) \subseteq E$, where $\{0\} \not\in E$. Hence the vanishing moment condition \eqref{eq:vanising} is satisfied. 
Now we only need to examine the decay properties of $\rho * \upsilon_{k,m}(u-\frac{\pi}{a}m)$. By definition
\begin{align*}
|(\rho * \upsilon_{k,m})(u-\frac{\pi}{a}m)| = \left|\int_{\mathbb{R}^d} \rho(u-\frac{\pi}{a}m - y) \upsilon_{k,m}(y) \d y \right|.
\end{align*}
Since $\hat{\rho}$ is constructed around $Q_j$ we use similar arguments as in the proof of Lemma \ref{lemma:1} and obtain the following estimate, see e.g. \cite[Appendix B]{MR3243741}.
%\begin{align*}
%|\rho(u)| \leq C t_j^{\nu/2} (1+t_j|u|_{\ca})^{-N}.
%\end{align*}
%Moreover, with the estimate for $\upsilon_{k,m}$, \eqref{eq:decayofpartialmu}, \footnote{Jeg er ikke helt sikker på $\upsilon_{k,m}$ opfylder denne betingelse? Nu hvor $\upsilon_{k,m}$ ikke længere } and by using the same technique as in the proof of Lemma \ref{lemma:1} we obtain
\begin{align}\label{eq:3}
|(\rho * \upsilon_{k,m})(u-\frac{\pi}{a}m)| &\leq C \int_{\mathbb{R}^d} \frac{t_j^{\nu/2}} {(1+t_j|u-\frac{\pi}{a}m-y|_{\ca})^{2N}} \frac{1}{(1+|y|_{\ca})^{2N}} \d y \nonumber \\
&\leq C t_j^{\nu/2} (1+t_j|u-\frac{\pi}{a}m|_{\ca})^{-2N}.
\end{align}
Now consider the integral in \eqref{eq:integralmedfoldning}. We evaluate this integral by using the same technique as in the proof of Lemma \ref{lemma:decay}. 
Set $f_{j,k,n}(u) = \mu_j(A_jA_k^{-1}u - \frac{\pi}{a}n)$ and $g_{k,m}(u) = \rho * \upsilon_{k,m}(u-\frac{\pi}{a}m)$. By the vanishing moments of the function $g_{k,m}(u)$ and by using the estimates \eqref{eq:1} and \eqref{eq:3} we obtain
\begin{align}\label{eq:fandgestimate}
&\left|\int_{\mathbb{R}^d} f_{j,k,n}(u)g_{k,m}(u) \d u \right| \leq 
\int_{\mathbb{R}^d} \left| f_{j,k,n}(u) - \sum_{|\gamma|\leq L-1} \frac{\partial^\gamma_{u} f_{j,k,n}(\frac{\pi}{a}m)}{\gamma !}(u-\frac{\pi}{a}m)^\gamma \right| | g_{k,m}(u)| \d u \nonumber \\
&\leq C \int_{\mathbb{R}^d} |u-\frac{\pi}{a}m|^L | \partial^\gamma_{u} f_{j,k,n}(y)||g_{k,m}(u)|\d u \nonumber \\
&\leq C \left(\frac{t_j}{t_k} \right)^{\alpha_1L} \int_{\mathbb{R}^d} \frac{\max\{|u-\frac{\pi}{a}m|_{\ca}^{L/\alpha_1}, |u-\frac{\pi}{a}m|_{\ca}^{L/\alpha_2}\}} {(1+\frac{t_j}{t_k}|y-u_{j,k,n}|_{\ca})^{2N}} \frac{t_j^{\nu/2}}{(1+t_j|u-\frac{\pi}{a}m|_{\ca})^{2N}} \d u, 
\end{align}
for some $y$ on the line segment joining $u$ and $\frac{\pi}{a}m$. Using that $t_j\leq t_k$ and $t_k\geq 1$, together with the quasi-triangle inequality, we find that
\begin{align}\label{eq:estimaty}
\frac{1}{K} \frac{1}{1+\frac{t_j}{t_k}|y-u_{j,k,n}|_{\ca}} \leq K \frac{1+t_j|u-\frac{\pi}{a}m|_{\ca}}{1+\frac{t_j}{t_k}|u_{j,k,n}-\frac{\pi}{a}m|_{\ca}}.
\end{align}

With this estimate we proceed from \eqref{eq:fandgestimate}.
\begin{align*}
&\left|\int_{\mathbb{R}^d} f_{j,k,n}(u)g_{k,m}(u) \d u \right| \\
&\leq C \left(\frac{t_j}{t_k}\right)^{\alpha_1L} \frac{t_j^{\nu/2}}{(1+\frac{t_j}{t_k}|u_{j,k,n}-\frac{\pi}{a}m|_{\ca})^{N}} \int_{\mathbb{R}^d} \frac{\max\{|u-\frac{\pi}{a}m|_{\ca}^{L/\alpha_1}, |u-\frac{\pi}{a}m|_{\ca}^{L/\alpha_2}\} }
{(1+t_j|u-\frac{\pi}{a}m|_{\ca})^{2N-N}} \d u \\
&\leq C \left(\frac{t_j}{t_k}\right)^{\alpha_1L} \frac{t_j^{\nu/2}}{(1+\frac{t_j}{t_k}|u_{j,k,n}-\frac{\pi}{a}m|_{\ca})^{N}} \int_{\mathbb{R}^d}
\frac{\max\{(t_j^{-1}|w|_{\ca})^{L/\alpha_1}, (t_j^{-1}|w|_{\ca})^{L/\alpha_2} \}}{(1+|w|_{\ca})^{N}} t_j^{-\nu} \d w \\
&\leq C \left(\frac{t_j}{t_k}\right)^{\alpha_1L} 
\frac{t_j^{-\nu/2-L/\alpha_1}}{(1+\frac{t_j}{t_k}|u_{j,k,n}-\frac{\pi}{a}m|_{\ca})^{N}} \int_{\mathbb{R}^d}
\frac{\max\{|w|_{\ca}^{L/\alpha_1}, |w|_{\ca}^{L/\alpha_2} \}}{(1+|w|_{\ca})^{N}} \d w \\
&\leq C \left(\frac{t_j}{t_k} \right)^{\alpha_1 L - \nu/2 - L/\alpha_1} (1+t_j|x_{j,n}-x_{k,m}|_{\ca})^{-N},
\end{align*}
where the last inequality follows since $N>  \nu + L/\alpha_1$.
Using this estimate in \eqref{eq:integralmedfoldning} we obtain
\begin{align}\label{eq:trouble}
\left|\ip{\eta_{j,n},\psi_{k,m}}\right| \leq C \left(\frac{t_j}{t_k} \right)^{\alpha_1 L - L/\alpha_1} (1+t_j|x_{j,n}-x_{k,m}|_{\ca})^{-N}.
\end{align}
The case $t_j\leq t_k < 1$ is handled in a similar fashion starting from \eqref{eq:startproof}, but without any change of variable. The details are left for the reader.
\end{proof}

\begin{myre}\label{re:scale}
It is precisely the exponent $(\alpha_1-\alpha_1^{-1})L$ appearing in \eqref{eq:trouble} that motivates our standing assumption that $\alpha_1>1$ to ensure that $\alpha_1-\alpha_1^{-1}>0$.
\end{myre}
We are now ready to state the following result, which is analogous to Proposition \ref{prop:almostdiag}, but with the  improvement  that we do not assume compact support in the frequency domain of the system $\{\psi_{k,m}\}_{k,m}$.

\begin{myprop}\label{prop:semigenerelle}
Let $K>0$ and $L\in \N$ satisfy $\nu/2+K = \frac{1}{2}(\nu/2 + \alpha_1 L - L/\alpha_1)$ with $N,M,K$ chosen such that $N>\nu+L/\alpha_1$ and $2M > \nu$. 
Let $\{\eta_{j,n}\}_{j\in J,n\in\Z^d}$ be the frame defined in \eqref{eq:eta} satisfying \eqref{eq:etadecay} and \eqref{eq:hatetadecay}, and let $\{\psi_{k,m}\}_{k\in J,m\in Z^d}$ satisfy \eqref{eq:psidecay} and \eqref{eq:hatpsidecay} with $N'>2N+L/\alpha_1$ and $M'> M+\frac{L}{\beta\alpha_1}$. Then there exists a constant $C:=C(K)>0$ such that
\begin{align*}
\left|\ip{\eta_{j,n},\psi_{k,m}} \right| &\leq C\min\left(\frac{t_j}{t_k},\frac{t_k}{t_j} \right)^{\nu/2+K} (1+\max(t_j,t_k)^{-1}|\xi_j-\xi_k|_{\ca})^{-M} \nonumber \\
&\times (1+\min(t_j,t_k)|x_{j,n}-x_{k,m}|_{\ca})^{-N/2}.
\end{align*}
\end{myprop}

\begin{proof}
We first notice that in case $\min(t_j,t_k) \geq 1$, we may conclude by using the result in Proposition \ref{prop:almostdiag}, since this particular case did not use the assumption of compact support. For the case $\min(t_j,t_k) < 1$ we assume,  without loss of generality that $t_j\leq t_k$.
Using Lemma \ref{lemma:enikkebandlimited} gives
\begin{align*}
\left|\ip{\eta_{j,n},\psi_{k,m}} \right| &\leq C \left(\frac{t_j}{t_k} \right)^{\alpha_1 L - L/\alpha_1} (1+t_j|x_{j,n}-x_{k,m}|_{\ca})^{-N}.
\end{align*}
Moreover, using Lemma \ref{lemma:1} for $\langle\hat{\eta}_{j,n},\hat{\psi}_{k,m}\rangle$ gives
\begin{align*}
|\langle\hat{\eta}_{j,n},\hat{\psi}_{k,m}\rangle| \leq C \left(\frac{t_j}{t_k}\right)^{\nu/2} (1+t_k^{-1}|\xi_j-\xi_k|_{\ca})^{-2M}.
\end{align*}
Now, inserting the above estimates in \eqref{eq:sqrt} we obtain
\begin{align*}
\left|\ip{\eta_{j,n},\psi_{k,m}} \right| &\leq C \left(\frac{t_j}{t_k} \right)^{\frac{1}{2}(\nu/2 + \alpha_1 L - L/\alpha_1)} (1+t_k^{-1}|\xi_j-\xi_k|_{\ca})^{-M} \\
&\times (1+t_j|x_{j,n}-x_{k,m}|_{\ca})^{-N/2}.
\end{align*}
Since $K+\nu/2 = \frac{1}{2}(\nu/2 + \alpha_1 L - L/\alpha_1)$ we have obtained the wanted estimate.  
The case $t_k\leq t_j$ and $t_k < 1$ follows in parallel with the above, and we therefore leave the details for the reader.  
\end{proof}

Comparing the result in Proposition \ref{prop:almostdiag} with the above we see that the matrix $\{\ip{\eta_{j,n},\psi_{k,m}}\}_{k,m,j,n}$ satisfies Definition \ref{def:almostdiagonal}, even though the assumptions about compact support for the functions $\psi_{k,m}$ were omitted. 
We now consider much more  general families of functions for which Definition \ref{def:almostdiagonal} hold. 

\begin{mytheorem}\label{thm:almostdiagonal}
Let $K>0$ and $L\in \N$ satisfy $\nu/2+K = \frac{1}{2}(\nu/2 + \alpha_1 L - L/\alpha_1)$ with $N,M,K$ chosen such that $N>\nu+L/\alpha_1$ and $2M > \nu$ and suppose $\{\psi^{(1)}_{j,n}\}_{j\in J, n\in\Z^d}$ and $\{\psi^{(2)}_{k,m}\}_{k\in J, m\in\Z^d}$ satisfy \eqref{eq:psidecay} and \eqref{eq:hatpsidecay}  with $N'>2N+L/\alpha_1$ and $M'> M+\frac{L}{\beta\alpha_1}$. Then there exists a constant $C:=C(K)>0$ such that
\begin{align*}
|\langle \psi^{(1)}_{j,n},\psi^{(2)}_{k,m}\rangle| &\leq C \min\left( \frac{t_j}{t_k},\frac{t_k}{t_j} \right)^{\nu/2+K} (1+\max(t_j,t_k)^{-1}|\xi_j-\xi_k|_{\ca})^{-M} \\
&\times (1+\min(t_j,t_k)|x_{j,n}-x_{k,m}|_{\ca})^{-N/2}.
\end{align*}
In particularly, suppose $s\in\R$, $0<p,q<\infty$ and put $r:=\min(1,p,q)$.  If \begin{equation}\label{eq:Nprime}
N'>\frac{2\nu}{r}+\frac{\nu r+4|s| r +4\nu}{r(\alpha_1^2-1)},
\end{equation} and 
\begin{equation}\label{eq:Mprime}
M'>\frac{\nu}{r}++\frac{L}{\beta\alpha_1}>\frac{\nu}{2}+\frac{\nu r+4|s| r +4\nu}{2r\beta(\alpha_1^2-1)},
\end{equation}
then $\{\langle \psi^{(1)}_{j,n},\psi^{(2)}_{k,m}\rangle\} \in \text{ad}_{p,q}^s$. 
\end{mytheorem}

\begin{proof}
Since $\{\eta_{i,l}\}_{i\in J, l\in \Z^d}$ is a tight frame we have
\begin{align*}
\langle\psi^{(1)}_{j,n},\psi^{(2)}_{k,m}\rangle = \sum_{i\in J,l\in\Z^d} \langle\psi^{(1)}_{j,n},\eta_{i,l}\rangle \langle\eta_{i,l},\psi^{(2)}_{k,m}\rangle.
\end{align*}
This corresponds to the composition of two operators with matrices  $\{\langle\psi^{(1)}_{j,n},\eta_{i,l}\rangle\}_{i,k,l,m}$ and $\{\langle\eta_{i,l},\psi^{(2)}_{k,m}\rangle\}_{m,l,i,k}$, respectively. By Proposition \ref{prop:semigenerelle} these matrices are almost diagonal. Proposition \ref{prop:lukket} now implies that the product of two almost diagonal matrices is almost diagonal. The final claim follows from the  estimates

$$K>|s|+2\frac{\nu}{r}\Rightarrow L>\frac{1}2\frac{(\nu r+4|s| r +4\nu)\alpha_1}{r(\alpha_1^2-1)},$$ so
$$N'>2\nu+2L/\alpha_1>2\nu+\frac{\nu r+4|s| r +4\nu}{r(\alpha_1^2-1)}.$$
and
$$M'>\frac{\nu}2+\frac{L}{\beta\alpha_1}>\frac{\nu}{2}+\frac{\nu r+4|s| r +4\nu}{2r\beta(\alpha_1^2-1)}.$$
At the same time we must have $N/2>\frac{\nu}{r}$ and $M>\frac{\nu}{r}$ by
 comparing the estimate of $|\langle \psi^{(1)}_{j,n},\psi^{(2)}_{k,m}\rangle| $ to Definition \ref{def:almostdiagonal}. This completes the proof. 

\end{proof}

\section{An application: Compactly supported frames}\label{sec:compactly}
We now turn to our main example of an application of the algebra of almost diagonal matrices. We will construct a system $\{\psi_{k,n}\}_{k\in J, n\in\Z^d}$ which is a small
perturbation of the frame $\{\eta_{k,n}\}_{k\in J, n\in\Z^d}$ given by
\eqref{eq:eta}.
%However, any other frame which satisfies
%\eqref{one} and \eqref{three} for appropriate $N,M,L$ can be used instead.
Following a general approach introduced by Kyriazis and Petrushev \cite{MR2204289} for classical Triebel-Lizorkin and Besov spaces, we  first to show that a system $\{\psi_{k,n}\}_{k\in J, n\in\Z^d}$,
which is close enough to the tight frame $\{\eta_{k,n}\}_{k\in J, n\in\Z^d}$, in a suitable sense, is also a frame for
$\dot{F}^s_{p,q}(\tilde{h})$ and $\dot{M}^s_{p,q}(\tilde{h})$. Next, to get a frame expansion in $\dot{F}^s_{p,q}(\tilde{h})$ and $\dot{M}^s_{p,q}(\tilde{h})$, we show that
$\{S^{-1}\psi_{k,n}\}_{k\in J, n\in\Z^d}$ is also a frame, where $S$ is the corresponding frame
operator given by
\begin{equation*}
 S f=\sum_{k\in J, n\in\Z^d}\langle f,\psi_{k,n}\rangle
\psi_{k,n}.
\end{equation*}

Now suppose $\{\psi_{k,n}\}_{k\in J, n\in\Z^d}\subset L_2(\R^d)$ is a system that is close to $\{\eta_{k,n}\}_{k\in J, n\in\Z^d}$ in
the sense that for fixed $s\in \R$ there exists $\varepsilon, \delta
>0$ such that
\begin{align}
&|\eta_{k,n}(x)-\psi_{k,n}(x)|\le \varepsilon t_k^{\frac{\nu}{2}}(1+t_k|x_{k,n}-x|_{\mathbf{a}})^{-2N'}\label{bruce},\\
&|\hat{\eta}_{k,n}(\xi)-\hat{\psi}_{k,n}(\xi)|\le \varepsilon
t_k^{-\frac{\nu}{2}}(1+t_k^{-1}|\xi_{k}-\xi|_{\mathbf{a}})^{-2M'},\label{springsteen}
\end{align}
where we have used the notation from Definition \ref{def:almostdiagonal}, and $N', M'$ satisfy the conditions given by \eqref{eq:Nprime} and \eqref{eq:Mprime}, respectively. 
Motivated by the fact that $\{\eta_{k,n}\}_{k\in J, n\in\Z^d}$ is a tight frame for $L_2(\R^d)$, we formally define $\langle f,\psi_{j,m}\rangle$ as
\begin{equation}
\langle f,\psi_{j,m}\rangle:=\sum_{k\in J}\sum_{n\in \Z^d}\langle \eta_{k,n},\psi_{j,m}\rangle \langle
f,\eta_{k,n}\rangle, \,\, f \in \dot{F}^s_{p,q}(\tilde{h}).
\end{equation}
We deduce from Proposition \ref{johnlennon} and Theorem \ref{thm:almostdiagonal} that $\langle \cdot,\psi_{j,m}\rangle$ is a bounded linear functional on $\dot{F}^s_{p,q}(\tilde{h})$; in fact we have
\begin{align}
\sum_{k\in J, n\in\Z^d}|\langle \eta_{k,n},\psi_{j,m}\rangle||\langle
f,\eta_{k,n}\rangle| \notag
&\le \Big\|\Big\{\sum_{k\in J, n\in\Z^d}|\langle \eta_{k,n},\psi_{j,m}\rangle| |\langle
f,\eta_{k,n}\rangle| \Big\}_{j,m\in\Z^d}\Big\|_{\dot{f}^s_{p,q}(\tilde{h})} \\ &\le C \|\langle
f,\eta_{k,n}\rangle \|_{\dot{f}^s_{p,q}(\tilde{h})}\le
C\|f\|_{\dot{F}^s_{p,q}(\tilde{h})}.\label{thiswillbeday}
\end{align}
Furthermore, $\{\psi_{k,n}\}_{k\in J, n\in\Z^d}$ is a norming family for $\dot{F}^s_{p,q}(\tilde{h})$ as it satisfies $$\|\langle f,\psi_{k,n}\rangle\|_{\dot{f}^s_{p,q}(\tilde{h})} \le C \|f\|_{\dot{F}^s_{p,q}(\tilde{h})}.$$ This can be used to show that $S$ is a bounded operator on $\dot{F}^s_{p,q}(\tilde{h})$, and for small enough $\varepsilon$, this will be the key to showing that $\{\psi_{k,n}\}_{k\in J, n\in\Z^d}$ is a frame for $\dot{F}^s_{p,q}(\tilde{h})$.
\begin{mytheorem}\label{generation1}%
There exists $\varepsilon_0,C_1,C_2>0$ such that if $\{\psi_{k,n}\}_{k\in J, n\in\Z^d}$ satisfies \eqref{bruce} and
\eqref{springsteen} for some $0<\varepsilon \le \varepsilon_0$ and $f\in \dot{F}^s_{p,q}(\tilde{h})$, then we have
\begin{equation}
C_1\|f\|_{\dot{F}^s_{p,q}(\tilde{h})}\le \|\langle
f,\psi_{k,n}\rangle\|_{\dot{f}^s_{p,q}(\tilde{h})}\le C_2\|f\|_{\dot{F}^s_{p,q}(\tilde{h})}.
\end{equation}
Similarly for $\dot{M}^s_{p,q}(\tilde{h})$ and $\dot{m}^s_{p,q}(\tilde{h})$.
\end{mytheorem}
\begin{proof}
The proof will only be given for $\dot{F}^s_{p,q}(\tilde{h})$ as it follows the
same way for $\dot{M}^s_{p,q}(\tilde{h})$. That $\{\psi_{k,n}\}_{k\in J, n\in\Z^d}$ is a norming
family gives the upper bound, thus we only need to establish the
lower bound. For this we notice that
$\{\varepsilon^{-1}(\eta_{k,n}-\psi_{k,n})\}_{k\in J, n\in\Z^d}$ is also a norming family
so we have
\begin{equation*}
\|\langle f,\eta_{k,n}-\psi_{k,n}\rangle\|_{\dot{f}^s_{p,q}(\tilde{h})}\le C
\varepsilon \|f\|_{\dot{F}^s_{p,q}(\tilde{h})}.
\end{equation*}
It then follows that
\begin{align*}
\|f\|_{\dot{F}^s_{p,q}(\tilde{h})} &\le C \|\langle
f,\eta_{k,n}\rangle\|_{\dot{f}^s_{p,q}(\tilde{h})}\\
&\le C (\|\langle f,\psi_{k,n}\rangle\|_{\dot{f}^s_{p,q}(\tilde{h})}+\|\langle
f,\eta_{k,n}-\psi_{k,n}\rangle\|_{\dot{f}^s_{p,q}(\tilde{h})}) \\
&\le C(\|\langle f,\psi_{k,n}\rangle\|_{\dot{f}^s_{p,q}(\tilde{h})}+\varepsilon
\|f\|_{\dot{F}^s_{p,q}(\tilde{h})}).
\end{align*}
By choosing $\varepsilon<1/C$ we get the lower bound. \\
\end{proof}
As suggested by Theorem \ref{generation1}, the boundedness
of the matrix $$\{\langle \eta_{k,n}, S^{-1}\psi_{j,m}\rangle\}_{k,j\in J; n,m\in\Z^d}$$ on
$\dot{f}^s_{p,q}(\tilde{h})$ is the key to showing that $\{S^{-1}\psi_{k,n}\}_{k\in J, n\in\Z^d}$ is
also a frame.
\begin{myprop}\label{layla}%
There exists $\varepsilon_0>0$ such that if $\{\psi_{k,n}\}_{k\in J, n\in\Z^d}$ is a frame for $\dot{F}^0_{22}(\tilde{h})=L_2(\R^d)$
and satisfies \eqref{bruce} and \eqref{springsteen} for some $0<\varepsilon\le\varepsilon_0$, then $\{\langle \eta_{k,n},
S^{-1}\psi_{j,m}\rangle\}_{k,j\in J; n,m\in\Z^d}$ is bounded on $\dot{f}^s_{p,q}(\tilde{h})$ and
$\dot{m}^s_{p,q}(\tilde{h})$.
\end{myprop}
The proof is identical to the proof of Proposition 5.2 in \cite{NR12} and we will therefore omit it.

The fact that $\{S^{-1}\psi_{k,n}\}_{k\in J, n\in\Z^d}$ is a frame for $\dot{F}^s_{p,q}(\tilde{h})$ and $\dot{M}^s_{p,q}(\tilde{h})$ now follows as a consequence of $\{\langle \eta_{k,n},
S^{-1}\psi_{j,m}\rangle\}_{k,n,j,m\in\Z^d}$ being bounded on $\dot{f}^s_{p,q}(\tilde{h})$ and $\dot{m}^s_{p,q}(\tilde{h})$.
We state the following results without proofs as they follow
directly in the same way as in the classical Triebel-Lizorkin and Besov spaces. The proofs
can be found in \cite{MR1862566}. First, we have the frame expansion.
\begin{mylemma}\label{mualim1}%
Assume that $\{\psi_{k,n}\}_{k\in J, n\in\Z^d}$ is a frame for $L_2(\R^d)$ and
satisfies
\begin{align}
&|\psi_{k,n}(x)|\le C t_k^{\frac{\nu}{2}}(1+t_k|x_{k,n}-x|_{\mathbf{a}})^{-2N'}, \label{easy1}\\
&|\hat{\psi}_{k,n}(\xi)|\le C
t_k^{-\frac{\nu}{2}}(1+t_k^{-1}|\xi_{k}-\xi|_{\mathbf{a}})^{-2M'}, \label{easy2}
\end{align}
where  $N', M'$ satisfy the conditions given by \eqref{eq:Mprime} and \eqref{eq:Nprime}, respectively.
If \\ $\{\langle \eta_{k,n}, S^{-1}\psi_{j,m}\rangle\}_{k,j\in J; n,m\in\Z^d}$ is bounded on
$\dot{f}^s_{p,q}(\tilde{h})$, then for $f\in \dot{F}^s_{p,q}(\tilde{h})$ we have
\begin{equation*}
f=\sum_{k\in J, n\in\Z^d}\langle f,S^{-1}\psi_{k,n}\rangle \psi_{k,n}
\end{equation*}
in the sense of $\mathcal{S}'/\mathcal{P}$. Similarly for $\dot{M}^s_{p,q}(\tilde{h})$ and $\dot{m}^s_{p,q}(\tilde{h})$.\\
\hspace*{\fill}\nolinebreak$\square$
\end{mylemma}
Moreover, we have that $\{S^{-1}\psi_{k,n}\}_{k\in J, n\in\Z^d}$ is a frame. The proof of the following fundamental result can easily be adapted from the technique introduced in \cite{MR2204289}.
\begin{mytheorem}\label{mualim2}%
Assume that $\{\psi_{k,n}\}_{k\in J, n\in\Z^d}$ is a frame for $L_2(\R^d)$ and
satisfies \eqref{easy1} and \eqref{easy2}.
%\begin{align*}
%&|\psi_{k,n}(x)|\le C
%t_k^{\frac{\nu}{2}}(1+t_k|x_{k,n}-x|_{\mathbf{a}})^{-2\left(\frac{\nu}{r}+\delta\right)},\\
%&|\hat{\psi}_{k,n}(\xi)|\le C
%t_k^{-\frac{\nu}{2}}(1+t_k^{-1}|\xi_{k}-\xi|_{\mathbf{a}})^{-2\left(\frac{\nu}{r}+\delta\right)-\frac{2}{\beta}\left(|s|+\frac{2\nu}{r}+\frac{3\delta}{2}\right)},
%\end{align*}
%where we have used the notation from Definition \ref{doitmad}.
Then $\{S^{-1}\psi_{k,n}\}_{k\in J, n\in\Z^d}$ is a frame for $\dot{F}^s_{p,q}(\tilde{h})$ if and
only if $\{\langle \eta_{k,n}, S^{-1}\psi_{j,m}\rangle\}_{k,j\in J; n,m\in\Z^d}$ is bounded
on $\dot{f}^s_{p,q}(\tilde{h})$. Similarly for $\dot{M}^s_{p,q}(\tilde{h})$ and $\dot{m}^s_{p,q}(\tilde{h})$.\\
%\hspace*{\fill}\nolinebreak$\square$
\end{mytheorem}
It is worth noting that Proposition \ref{layla}, Lemma \ref{mualim1} and Theorem \ref{mualim2} imply that $\{\psi_{k,n}\}_{k\in J, n\in\Z^d}$ is a Banach frame if it satisfies \eqref{bruce} and \eqref{springsteen} with sufficiently small $\varepsilon$, and $p,q\ge 1$. Furthermore, following a similar approach we can obtain a frame expansion with $\{S^{-1}\psi_{k,n}\}_{k\in J, n\in\Z^d}$.
\begin{mylemma}\label{mualim3}%
Assume that $\{\psi_{k,n}\}_{k\in J, n\in\Z^d}$ is a frame for $L_2(\R^d)$ and
satisfies \eqref{easy1} and \eqref{easy2}.
%\begin{align*}
%&|\psi_{k,n}(x)|\le C %t_k^{\frac{\nu}{2}}(1+t_k|x_{k,n}-x|_{\mathbf{a}})^{-2\left(\frac{\nu}{r}+\delta\right)},\\
%&|\hat{\psi}_{k,n}(\xi)|\le C
%t_k^{-\frac{\nu}{2}}(1+t_k^{-1}|\xi_{k}-\xi|_{\mathbf{a}})^{-2\left(\frac{\nu}{r}+\delta\right)-\frac{2}{\beta}\left(|s|+\frac{2\nu}{r}+\frac{3\delta}{2}\right)},
%\end{align*}
%where we have used the notation from Definition \ref{doitmad}.
If the transpose of $\{\langle \eta_{k,n}, S^{-1}\psi_{j,m}\rangle\}_{k,j\in J, n,m\in\Z^d}$ is bounded on $\dot{f}^s_{p,q}(\tilde{h})$, then for $f\in \dot{F}^s_{p,q}(\tilde{h})$ we have
\begin{equation*}
f=\sum_{k\in J, n\in\Z^d}\langle f,\psi_{k,n}\rangle S^{-1}\psi_{k,n}
\end{equation*}
in the sense of $\mathcal{S}'\backslash \mathcal{P}$. Similarly for $\dot{M}^s_{p,q}(\tilde{h})$ and $\dot{m}^s_{p,q}(\tilde{h})$.\\
%\hspace*{\fill}\nolinebreak$\square$
\end{mylemma}

%\begin{myre}
%If we have that $\{\eta_{k,n}\}_{k\in J, n\in\Z^d}$ is normalized in $L_2(\R^d)$, then $\{\eta_{k,n}\}_{k\in J, n\in\Z^d}$ is an orthonormal basis for $L_2(\R^d)$ as a consequence of $\{\eta_{k,n}\}_{k\in J, n\in\Z^d}$ being a tight frame for $L_2(\R^d)$ with constant 1. With arguments similar to the ones used in the proof of Proposition \ref{layla}, it can be shown that there exists $\varepsilon_0$ such that if $\{\psi_{k,n}\}_{k\in J, n\in\Z^d}$ satisfies \eqref{bruce} and \eqref{springsteen} for some $\varepsilon\le \varepsilon_0$, then $\{\langle \eta_{k,n},\psi_{j,m}\rangle\}_{k,j\in J;n,m\in\Z^d}$ has a bounded inverse on $\dot{f}^s_{p,q}(\tilde{h})$ and $\dot{m}^s_{p,q}(\tilde{h})$, and consequently $\{\psi_{k,n}\}_{k\in J, n\in\Z^d}$ is an unconditional basis for $\dot{F}^s_{p,q}(\tilde{h})$ and $\dot{M}^s_{p,q}(\tilde{h})$.\\
%\end{myre}

In particular, by using a generating function $g$ with compact support one can construct a compactly supported frame expansion. A successful approach to problems of this type, see e.g.\ \cite{NR12,MR2204289,MR2891260}, is to use finite linear combinations of a function with sufficient smoothness and decay in direct space and vanishing moments. \\
\indent In general, it suffices to obtain a system of functions $\{\tau_k\}_{k\in\Z^d}\subset L_2(\R^d)$ which is close enough to $\{\mu_k\}_{k\in\Z^d}$,
\begin{align*}
|\mu_k(x)-\tau_k(x)|&\le \varepsilon(1+|x|_{\mathbf{a}})^{-2N'},\\
|\hat{\mu}_k(\xi)-\hat{\tau}_k(\xi)|&\le \varepsilon(1+|\xi|_{\mathbf{a}})^{-2M'}.
\end{align*}
 The system
\begin{equation*}
\{\psi_{k,n}\}_{k\in J, n\in\Z^d}:=\Big\{t_k^{\nu/2}\tau_k\Big(A_kx-\frac{\pi}{a}n\Big)e^{ix
\cdot \xi_k}\Big\}_{k\in J, n\in\Z^d}
\end{equation*}
will then satisfy \eqref{bruce} and \eqref{springsteen}. First, we take $g\in C^{1}(\R^d)\cap L_2(\R^d)$, $\hat{g}(0)\not= 0$, which for fixed $N'',M''>0$ satisfies
\begin{align}
|g^{(\kappa)}(x)|&\le C(1+|x|_{\mathbf{a}})^{-2N''}, \,\,
|\kappa|\le 1, \label{kappastaff1} \\
|\hat{g}(\xi)|&\le C(1+|\xi|_{\mathbf{a}})^{-2M''}.
\label{kappastaff2}
\end{align}
Next for $m\ge 1$, we define $g_m(x):=C_g {m}^\nu g(D_{\mathbf{a}}(m) x)$,
where $C_g:=\hat{g}(0)^{-1}$.
 %It then follows that
%\begin{align}
%|g_m^{(\kappa)}(x)|&\le C
%{m}^{\nu+\alpha_2|\kappa|}(1+m|x|_{\mathbf{a}})^{-N-\alpha_1},\,\,
% |\kappa|\le 1, \notag\\
%&\phantom{C}\int_{\R} g_m(x) \d x = 1 \label{gstar2},\\
%|\hat{g}_m(\xi)|&\le C {m}^{M+\alpha_2}(1+|\xi|_{\mathbf{a}})^{-M-\alpha_2}.\notag
%\end{align}
To construct $\tau_k$ we will use the following set of finite linear combinations,
\begin{equation*}
\Theta_{K,m}=\{\psi : \psi(\cdot)=\sum_{i=1}^K a_i
g_m(\cdot+b_i),a_i\in \mathbb{C}, b_i\in \mathbb{R}^d\}.
\end{equation*}
The following result proved in \cite{NR12} provides us with the function we need.
\begin{myprop}\label{notmylove}%
Let $N''>N'>\nu$ and $M''>M'>\nu$. If $g\in C^{1}(\R^d)\cap L_2(\R^d)$, $\hat{g}(0)\not= 0$, fulfills \eqref{kappastaff1} and \eqref{kappastaff2} and $\mu_k\in C^{1}(\R^d)\cap L_2(\R^d)$ fulfills
\begin{align*}
|\mu_k(x)|&\le C(1+|x|_{\mathbf{a}})^{-2N''}, \\
|\mu_k^{(\kappa)}(x)|&\le C, \, |\kappa|\le 1,
\\ |\hat{\mu}_k(\xi)|&\le C (1+|\xi|_{\mathbf{a}})^{-2M''},
\end{align*}
then for any $\varepsilon>0$ there exists $K,m \ge 1$ and $\tau_k \in
\Theta_{K,m}$ such that
\begin{align}
|\mu_k(x)-\tau_k(x)|&\le \varepsilon (1+|x|_{\mathbf{a}})^{-2N'}, \label{lostsomeone1} \\
|\hat{\mu}_k(\xi)-\hat{\tau}_k(\xi)|&\le \varepsilon (1+|\xi|_{\mathbf{a}})^{-2M'}. \label{lostsomeone2}
\end{align}
\end{myprop}
We conclude this paper with the following direct consequence of  Theorem \ref{mualim2}, Lemma \ref{mualim3}, and Proposition \ref{notmylove}.
\begin{mycor}
	Choose $s\in\R$, $0<p<\infty$, and $0<q<\infty$.  Let $N', M'$ satisfy the conditions given by \eqref{eq:Nprime} and \eqref{eq:Mprime}, respectively, and pick $N''>N'>\nu$ and $M''>M'>\nu$.
	 If $g\in C^1(\R^d)\cap L_2(\R^d)$, $\hat{g}(0)\not=0$, satisfies
	\begin{align*}
	|g^{(\kappa)}(x)|&\le C(1+|x|)^{-2N''}, \, \, |\kappa|\le 1, \\
	|\hat{g}(\xi)|&\le C(1+|x|)^{-2M''},
	\end{align*}
	then there exists $K\in\N$ and $\psi_{k,n}(x):=e^{ix\cdot d_k}\sum_{i=1}^{K}a_{k,i}g(c_{k}x+b_{k,n,i})$,  $a_{k,i}\in\C$, $b_{k,n,i},d_k\in\R^d$, $c_{k}\in \R$, such that $\{S^{-1}\psi_{k,n}\}_{k\in J,n\in\Z^d}$ constitutes a frame for $\dot{F}^s_{p,q}(\tilde{h})$  and
	\begin{equation*}
	f=\sum_{k\in J ,n\in\Z^d}\langle f, S^{-1}\psi_{k,n}\rangle \psi_{k,n}
	\end{equation*}
	for all $f\in \dot{F}^s_{p,q}(\tilde{h})$ with convergence in $\mathcal{S}'\backslash \mathcal{P}$. A similar result holds for $\dot{M}^s_{p,q}(\tilde{h})$.  
%	\hspace*{\fill}\nolinebreak$\square$
\end{mycor}
\appendix
\section{Some addition results and technical proofs}
This appendix contains a number of additional results and various technical proofs.
\begin{myprop}\label{johnlennon}%
Suppose that $A \in \textrm{ad}_{p,q}^s(\tilde{h})$. Then $A$ is bounded on
$\dot{f}^s_{p,q}(\tilde{h})$ and $\dot{m}^s_{p,q}(\tilde{h})$.
\end{myprop}
\begin{proof}[Proof of Proposition \ref{johnlennon}]
We only prove the result for $\dot{f}^s_{p,q}(\tilde{h})$ when $q<\infty$ as
$q=\infty$ follows in a similar way with $l_q$ replaced by $l_\infty$, and
the proof for $\dot{m}^s_{p,q}(\tilde{h})$ is similar to the one for $\dot{f}^s_{p,q}(\tilde{h})$. Let $s:=\{s_{k,n}\}_{k\in J, n\in\Z^d}\in
\dot{f}^s_{p,q}(\tilde{h})$ and assume for now that $p,q> 1$. We write
$\mathbf{A}:=\mathbf{A}_0+\mathbf{A}_1$ such that
\begin{equation*}
(\mathbf{A}_0s)_{(j,m)}\!=\!\!\sum_{k:t_k\ge t_j}\sum_{n\in\Z^d}
a_{(j,m)(k,n)}s_{k,n}\hspace{0.2cm}\textrm{and}\hspace{0.2cm}(\mathbf{A}_1s)_{(j,m)}\!=\!\!\sum_{k:t_k<
t_j}\sum_{n\in\Z^d} a_{(j,m)(k,n)}s_{k,n}.
\end{equation*}
By using Lemma \ref{rock} we have
\begin{align*}
|(\mathbf{A}_0s)_{(j,m)}|&\le C\sum_{k:t_k\ge
t_j}\left(\frac{t_k}{t_j}\right)^{s+\frac{\nu}{2}-\frac{\nu}{r}-\frac{\delta}{2}}
c_{jk}^\delta
\sum_{n\in\Z^d}\frac{|s_{k,n}|}{\left(1+t_j\left|x_{k,n}-x_{j,m}\right|_B\right)^{\frac{\nu}{r}+\delta}}\\
&\le C \sum_{k:t_k\ge
t_j}\left(\frac{t_k}{t_j}\right)^{s+\frac{\nu}{2}-\frac{\delta}{2}}
c_{jk}^\delta M_r^{\mathbf{a}}\Big(\sum_{n\in
\Z^d}|s_{k,n}|\chi_{Q(k,n)}\Big)(x),
\end{align*}
for $x\in Q(j,m)$, where $t_k$ is defined in \eqref{eq:tj}, $x_{k,n}$ in \eqref{eq:x}, $Q(j,m)$ in \eqref{eq:pointsets} and $M_r^{\mathbf{a}}$ in \eqref{eq:maximaloperator}. It then follows by H\"{o}lder's inequality and Lemma \ref{pinball} below that
\begin{align*}
\sum_{m\in\Z^d} |(\mathbf{A}_0&s)_{(j,m)}\chi_{Q(j,m)}|^q \le
C\bigg(\sum_{k:t_k\ge
t_j}\left(\frac{t_k}{t_j}\right)^{s+\frac{\nu}{2}}
c_{jk}^\delta M_r^{\mathbf{a}}\Big(\sum_{n\in
\Z^d}|s_{k,n}|\chi_{Q(k,n)}\Big)\bigg)^q \\
&\le C\sum_{k:t_k\ge
t_j}c_{jk}^\delta \bigg(\left(\frac{t_k}{t_j}\right)^{s+\frac{\nu}{2}}M_r^{\mathbf{a}}\Big(\sum_{n\in
\Z^d}|s_{k,n}|\chi_{Q(k,n)}\Big)\bigg)^q \bigg(\sum_{i:t_i\ge t_j}
c_{ji}^\delta \bigg)^{q-1}\\
&\le C\sum_{k:t_k\ge
t_j}c_{jk}^\delta\bigg(\left(\frac{t_k}{t_j}\right)^{s+\frac{\nu}{2}}M_r^{\mathbf{a}}\Big(\sum_{n\in
\Z^d}|s_{k,n}|\chi_{Q(k,n)}\Big)\bigg)^q.
\end{align*}
We obtain
\begin{align*}
\left\|\mathbf{A}_0s\right\|_{\dot{f}^s_{p,q}(\tilde{h})} &\le
C\bigg\|\bigg(\sum_{j\in\Z^d}\sum_{k:t_k\ge t_j}c_{jk}^\delta \bigg(t_k^{s+\frac{\nu}{2}} {M}^\mathbf{a}_r\Big(\sum_{n\in
\Z^d}|s_{k,n}|\chi_{Q(k,n)}\Big)\bigg)^q\bigg)^{1/q}\bigg\|_{L_p}\\
&\le C
\bigg\|\bigg(\sum_{k\in J}\bigg(t_k^{s+\frac{\nu}{2}}{M}^\mathbf{a}_r\Big(\sum_{n\in
\Z^d}|s_{k,n}|\chi_{Q(k,n)}\Big)\bigg)^q\bigg)^{1/q}\bigg\|_{L_p}.
\end{align*}
Using the vector-valued Fefferman-Stein maximal inequality
\eqref{eq:feffermanstein}, we arrive at
\begin{equation*}
\left\|\mathbf{A}_0s\right\|_{\dot{f}^s_{p,q}(\tilde{h})}\le
C\Big\|\Big(\sum_{k\in J,n\in
\Z^d}(t_k^{s+\frac{\nu}{2}}|s_{k,n}|)^q\chi_{Q(k,n)}\Big)^{1/q}\Big\|_{L_p}=C\left\|s\right\|_{\dot{f}^s_{p,q}(\tilde{h})}.
\end{equation*}
The corresponding estimate for $\mathbf{A}_1$ follows from
the same type of arguments resulting in both $\mathbf{A}_0$ and
$\mathbf{A}_1$ being bounded on $\dot{f}^s_{p,q}(\tilde{h})$ and thereby
$\mathbf{A}$. For the cases $q=1$ and $p\le 1, q>1$ choose $0<\tilde{r}<r$ and $0<\tilde{\delta}<\delta$ such that $\nu/r+\delta/2\ge\nu/\tilde{r}+\tilde{\delta}/2$ and repeat the argument with $r:=\tilde{r}$, and $\delta:=\tilde{\delta}$. The case $q<1$ follows from first observing that
\begin{equation*}
\tilde{\mathbf{A}}:=\{\tilde{a}_{(j,m)(k,n)}\}:=\bigg\{|a_{(j,m)(k,n)}|^q\left(\frac{t_k}{t_j}\right)^{\frac{\nu}{2}-\frac{\nu
q}{2}}\bigg\}
\end{equation*}
is almost diagonal on $\dot{f}^{sq}_{\frac{p}{q},1}(\tilde{h})$. Furthermore, if
$v:=\{v_{k,n}\}:=\{|s_{k,n}|^qt_k^{\frac{\nu q}{2}-\frac{\nu}{2}}\}$
we have
\begin{align*}
\|v\|^{\frac{1}{q}}_{\dot{f}^{sq}_{\frac{p}{q},1}(\tilde{h})} =
\Big\|\Big(\sum_{k\in J, n\in\Z^d}\big(t_k^{s+\frac{\nu}{2}}|s_{k,n}|\big)^q \chi_{Q(k,n)}\Big)^{1/q}\Big\|_{L_p}=
\|s\|_{\dot{f}^{s}_{p,q}(\tilde{h})}.
\end{align*}
Before we can put these two observations into use we need that
\begin{align*}
|(\mathbf{A}s)_{(j,m)}|^q\le\sum_{k\in J}\sum_{n\in\Z^d}|a_{(j,m)(k,n)}|^q|s_{k,n}|^q
=t_j^{\frac{\nu}{2}-\frac{\nu
q}{2}}\sum_{k\in J}\sum_{n\in\Z^d}\tilde{a}_{(j,m)(k,n)}v_{k,n}.
\end{align*}
We then have
\begin{align*}
\|\mathbf{A}s\|_{\dot{f}^s_{p,q}(\tilde{h})} \le
\|\tilde{\mathbf{A}}v\|^{\frac{1}{q}}_{\dot{f}^{s q}_{\frac{p}{q},1}(\tilde{h})}
\le C \|v\|^{\frac{1}{q}}_{\dot{f}^{s q}_{\frac{p}{q},1}(\tilde{h})}=
C\|s\|_{\dot{f}^{s}_{p,q}(\tilde{h})}.
\end{align*}
\end{proof}

%#####
%####
\begin{mylemma}\label{rock}%
Suppose that $0 < r \le 1$ and $N>\nu/r$. Then for any sequence
$\{s_{k,n}\}_{k\in J, n\in\Z^d}\subset\C$, and for $x\in Q(j,m)$, we have
\begin{align}
\sum_{n\in
\Z^d}\frac{|s_{k,n}|}{(1+\min(t_k,t_j)|x_{k,n}-x_{j,m}|_B)^N} \le &C
\max\left(\frac{t_k}{t_j},1\right)^{\frac{\nu}{r}}\notag \\
&\phantom{C}\times {M}^\mathbf{a}_r\Big(\sum_{n\in
\Z^d}|s_{k,n}|\chi_{Q(k,n)}\Big)(x),\label{setty}
\end{align}
with $t_k$ defined in \eqref{eq:tj}, $x_{k,n}$ in \eqref{eq:x}, and $Q(j,m)$ in \eqref{eq:pointsets}.
\end{mylemma}
\begin{proof}
Without loss of generality we may assume $x_{j,m}=0$ and begin by
considering the case $t_k\le t_j$. We define the sets,
\begin{align*}
S_0&=\{n\in \Z^d :t_k|x_{k,n}|_B\le 1\},\\
S_i&=\{n\in \Z^d :2^{i-1}<t_k|x_{k,n}|_B\le 2^i\},\,\,\, i\ge 1.
\end{align*}
Choose $x\in Q(j,m)$. There exists $C_1>0$ such that $\cup_{n\in S_i}Q(k,n) \subset B_{\mathbf{a}}(x,C_12^{i}t_k^{-1})$, and by using $\int \chi_{Q(k,n)}=\kappa_d^{\mathbf{a}}t_k^{-\nu}$, we get
\begin{align*}
\sum_{n\in S_i}\frac{|s_{k,n}|}{(1+t_k|x_{k,n}|)^N} & \le C 2^{-iN}\sum_{n\in S_i}|s_{k,n}|\le C 2^{-iN}\Big(\sum_{n\in S_i}|s_{k,n}|^r\Big)^{\frac{1}{r}} \\
 & \le C 2^{-iN}\bigg(\frac{t_k^{\nu}}{\kappa_d^{\mathbf{a}}}\int_{B_{\mathbf{a}}(x,C_12^{i}t_k^{-1})}\sum_{n\in S_i}|s_{k,n}|^r\chi_{Q(k,n)}\bigg)^{\frac{1}{r}}.
\end{align*}
Hence by the definition of the maximal operator \eqref{eq:maximaloperator} we have
\begin{align*}
\sum_{n\in S_i}\frac{|s_{k,n}|}{(1+t_k|x_{k,n}|)^N} &\le C
2^{i(\frac{\nu}{r}-N)}\bigg(\frac{t_k^{\nu}}{2^{i\nu}\kappa_d^{\mathbf{a}}}\int_{B_{\mathbf{a}}(x,C_12^{i}t_k^{-1})}\sum_{n\in
S_i}|s_{k,n}|^r\chi_{Q(k,n)}\bigg)^{\frac{1}{r}}\\
&\le C 2^{i(\frac{\nu}{r}-N)}{M}^\mathbf{a}_r\Big(\sum_{n\in
\Z^d}|s_{k,n}|\chi_{Q(k,n)}\Big)(x)
\end{align*}
by using $\sum_{n\in \Z^d}\chi_{Q(k,n)}\le n_0$. Summing over $i\ge
0$ and using $N>\nu/r$ gives \eqref{setty}. For the second case, $t_k
> t_j$, we redefine the sets,
\begin{align*}
S_0&=\{n\in \Z^d :t_j|x_{k,n}|_B\le 1\}\\
S_i&=\{n\in \Z^d :2^{i-1}<t_j|x_{k,n}|_B\le 2^i\}, i\ge 1.
\end{align*}
As before we have
\begin{align*}
\sum_{n\in S_i}\frac{|s_{k,n}|}{(1+t_j|x_{k,n}|)^M}\le& C
2^{-iN}\bigg(\frac{t_{k}^{\nu}}{\kappa_d^{\mathbf{a}}}\int_{B_{\mathbf{a}}(x,C_12^{i}t_j^{-1})}\sum_{n\in
S_i}|s_{k,n}|^r\chi_{Q(k,n)}\bigg)^{\frac{1}{r}} \\
\le& C
2^{i(\frac{\nu}{r}-N)}\left(\frac{t_k}{t_j}\right)^{\frac{\nu}{r}}M^\mathbf{a}_r\Big(\sum_{n\in
\Z^d}|s_{k,n}|\chi_{Q(k,n)}\Big)(x).
\end{align*}
Summing over $i\ge0$ again gives \eqref{setty}.\\
\end{proof}
\begin{mylemma}\label{pinball}%
Assume \eqref{eq:ekstrabetingelse} is satisfied, and let $\delta>0$. There exists $C>0$ independent of $k$ such that
\begin{equation*}
\sum_{j\in \Z^d} \min\bigg(\bigg(\frac{t_j}{t_k}\bigg)^{\nu},\bigg(\frac{t_k}{t_j}\bigg)^{\delta}\bigg)
(1+\max(t_k,t_j)^{-1}|\xi_j-\xi_k|_{\mathbf{a}})^{-\nu-\delta}\le C,
\end{equation*}
with $t_k$ and $\xi_k$ defined as in Definition \ref{def:almostdiagonal}.
\end{mylemma}
\begin{proof}
%Let us suppose $t_k\geq t_j$.
 We begin by dividing the indices into sets,
\begin{align*}
S_0&=\{j\in J:|\xi_j-\xi_k|_{\mathbf{a}}\le \rho_1 t_k\} \\
S_i&=\{j\in J:2^{i-1}\rho_1 t_k <|\xi_j-\xi_k|_{\mathbf{a}}\le 2^i \rho_1 t_k\},\,\,\, i\ge 1,
\end{align*}
with $\rho_1$ defined in \eqref{eq:ekstrabetingelse}. Next, we divide the sum even further
by first looking at $t_k\ge t_j$. For such $j\in S_i$,  we have $B_{\mathbf{a}}(\xi_j,t_j)
\subset B_{\mathbf{a}}(\xi_k,C_12^i t_k)$ which follows from using \eqref{eq:ekstrabetingelse}:
\begin{align*}
|\xi_k-\xi|_{\mathbf{a}} \le K (|\xi_k-\xi_j|_{\mathbf{a}}+|\xi_j-\xi|_{\mathbf{a}}) \leq& K(2^i\rho_1 t_k+t_j)\\
\leq& K(2^i\rho_1 t_k+R_1 2^i t_k)\\
=& C_1 2^i t_k,
\end{align*}
for $\xi\in B_{\mathbf{a}}(\xi_j,t_j)$. By using that the covering $\{B_{\mathbf{a}}(\xi_j,t_j)\}_j$ is admissible, we get
\begin{align*}
\sum_{\substack{j\in S_i \\ j: t_j\le t_k}} \bigg(\frac{t_j}{t_k}\bigg)^{\nu}& (1+t_k^{-1}|\xi_j-\xi_k|_{\mathbf{a}})^{-\nu-\delta}\\
\le & C2^{-i(\nu+\delta)} \sum_{\substack{j\in S_i \\ j: t_j\le t_k}} \bigg(\frac{t_j}{t_k}\bigg)^{\nu}\frac{1}{\kappa_d^{\mathbf{a}} t_j^\nu}\int_{B_{\mathbf{a}}(\xi_j,t_j)}\chi_{B_{\mathbf{a}}(\xi_j,t_j)}(\xi)\d \xi\\
\le & C2^{-i(\nu+\delta)}\frac{1}{\kappa_d^{\mathbf{a}} t_k^\nu}\int_{B_{\mathbf{a}}(\xi_k,C_12^i t_k)}\sum_{\substack{j\in S_i \\ j: t_j\le t_k}}\chi_{B_{\mathbf{a}}(\xi_j,t_j)}(\xi)\d \xi\\
\le & C2^{-i\delta}.
\end{align*}
Summing over $i$ gives the lemma for the $t_k\ge t_j$ part of the sum.
 In a similar way, the result for $t_k<t_j$ follows by using
\begin{align*}
\sum_{\substack{j\in S_i \\ j: t_j> t_k}} \bigg(\frac{t_k}{t_j}\bigg)^{\delta} (1+t_j^{-1}|\xi_j-\xi_k|_{\mathbf{a}})^{-\nu-\delta}
\le\sum_{\substack{j\in S_i \\ j: t_j> t_k}} \bigg(\frac{t_j}{t_k}\bigg)^{\nu} (1+t_k^{-1}|\xi_j-\xi_k|_{\mathbf{a}})^{-\nu-\delta}.
\end{align*}
\end{proof}

\subsection{Proof of Proposition \ref{prop:lukket}}
It turns out that the class of almost diagonal matrices is closed under composition. We now give a proof of this fact. For notational convenience, we let
\begin{align}
w^{s,\delta}_{(j,m)(k,n)}:=&\bigg(\frac{t_k}{t_j}\bigg)^{s+\frac{\nu}{2}}\min\bigg(\bigg(\frac{t_j}{t_k}\bigg)^{\frac{\nu}{r}+\frac{\delta}{2}},
\bigg(\frac{t_k}{t_j}\bigg)^{\frac{\delta}{2}}\bigg)c_{jk}^\delta \notag \\
& \phantom{C}\times
(1+\min(t_k,t_j)|x_{k,n}-x_{j,m}|_B)^{-\frac{\nu}{r}-\delta},\notag
\end{align}
where we have used the notation from Definition \ref{def:almostdiagonal}. The following result holds.
\begin{mylemma}\label{poison}%
Let $s\in \R$, $0<r\le 1$ and $\delta>0$. We then have
\begin{align*}
\sum_{i\in J,l\in \Z^d}
w^{s,\delta}_{(j,m)(i,l)}w^{s,\delta}_{(i,l)(k,n)}\le C
w^{s,\delta/2}_{(j,m)(k,n)},
\end{align*}
\end{mylemma}

It follows directly from Lemma \ref{poison} that for $\delta_1,\delta_2>0$ we have
\begin{equation}\label{smellsliketeen}
\sum_{i\in J,l\in \Z^d}
w^{s,\delta_1}_{(j,m)(i,l)}w^{s,\delta_2}_{(i,l)(k,n)}\le C
w^{s,\min(\delta_1,\delta_2)/2}_{(j,m)(k,n)}
\end{equation}
which proves that $\textrm{ad}_{p,q}^s(\tilde{h})$ is closed under composition, hence proving Proposition \ref{prop:lukket}.

\begin{proof}[Proof of Lemma \ref{poison}]
Notice that the factors $t_i^{s+\frac{\nu}{2}}$ in the first terms
of $w^{s,\delta}_{(j,m)(i,l)}$ and $w^{s,\delta}_{(i,l)(k,n)}$
cancel leaving $(t_k/t_j)^{s+\frac{\nu}{2}}$ which
can be moved outside the sums. Therefore we only need to deal with
the last three terms in $w^{s,\delta}_{(j,m)(i,l)}$ and
$w^{s,\delta}_{(i,l)(k,n)}$. First we consider the case $t_j\le t_k$
and split the sum over $i$ into three parts,
\begin{align*}
\sum_{i\in J,l\in \Z^d}
w^{s,\delta}_{(j,m)(i,l)}w^{s,\delta}_{(i,l)(k,n)}=&\left(\frac{t_k}{t_j}\right)^{s+\frac{\nu}{2}}\left(\sum_{i:t_i>t_k}+\sum_{i:t_j\le
t_i\le t_k}+\sum_{i:t_i<t_j}\right)\sum_{l\in\Z^d}\ldots\\
=&
\left(\frac{t_k}{t_j}\right)^{s+\frac{\nu}{2}}\left(\textrm{I}+\textrm{II}+\textrm{III}\right).
\end{align*}
For I, by using Lemma \ref{life1} and Lemma \ref{life2} below, we have
\begin{align*}
\textrm{I}=&\sum_{i:t_i>t_k}\sum_{l\in\Z^d}\left(\frac{t_j}{t_i}\right)^{\frac{\nu}{r}+\frac{\delta}{2}}
\left(\frac{t_k}{t_i}\right)^{\frac{\delta}{2}}c_{ji}^\delta c_{ik}^\delta\\
&\phantom{C}\times
\frac{1}{(1+t_j|x_{j,m}-x_{i,l}|_B)^{\frac{\nu}{r}+\delta}}\frac{1}{(1+t_k|x_{k,n}-x_{i,l}|_B)^{\frac{\nu}{r}+\delta}}\\
\le&\frac{C}{(1+t_j|x_{j,m}-x_{k,n}|_B)^{\frac{\nu}{r}+\delta}}\sum_{i:t_i>t_k}\left(\frac{t_j}{t_i}\right)^{\frac{\nu}{r}+\frac{\delta}{2}} \left(\frac{t_k}{t_i}\right)^{\frac{\delta}{2}-\nu}c_{ji}^\delta c_{ik}^\delta\\
\le&C\left(\frac{t_j}{t_k}\right)^{\frac{\nu}{r}+\frac{\delta}{2}}c_{jk}^{\delta/2}\frac{1}{(1+t_j|x_{j,m}-x_{k,n}|_B)^{\frac{\nu}{r}+\delta}}.
\end{align*}
Similarly for II we get
\begin{align*}
\textrm{II}=&\sum_{i:t_j\le t_i\le
t_k}\sum_{l\in\Z^d}\left(\frac{t_j}{t_i}\right)^{\frac{\nu}{r}+\frac{\delta}{2}}\left(\frac{t_i}{t_k}\right)^{\frac{\nu}{r}+\frac{\delta}{2}}
c_{ji}^\delta c_{ki}^\delta\\
&\phantom{C}\times\frac{1}{(1+t_j|x_{j,m}-x_{i,l}|_B)^{\frac{\nu}{r}+\delta}}\frac{1}{(1+t_i|x_{k,n}-x_{i,l}|_B)^{\frac{\nu}{r}+\delta}}\\
\le&C\left(\frac{t_j}{t_k}\right)^{\frac{\nu}{r}+\frac{\delta}{2}}c_{jk}^{\delta/2}\frac{1}{(1+t_j|x_{j,m}-x_{k,n}|)^{\frac{\nu}{r}+\delta}}.
\end{align*}
For III we get
\begin{align*}
\textrm{III}=&\sum_{i:t_i<t_j}\sum_{l\in\Z^d}\left(\frac{t_i}{t_j}\right)^{\frac{\delta}{2}}\left(\frac{t_i}{t_k}\right)^{\frac{\nu}{r}+\frac{\delta}{2}}
c_{ji}^\delta c_{ik}^\delta\\ &\phantom{C}\times
\frac{1}{(1+t_i|x_{j,m}-x_{i,l}|_B)^{\frac{\nu}{r}+\delta}}\frac{1}{(1+t_i|x_{k,n}-x_{i,l}|_B)^{\frac{\nu}{r}+\delta}}\\
\le&\sum_{i:t_i<t_j}C\left(\frac{t_i}{t_j}\right)^{\frac{\delta}{2}}\left(\frac{t_i}{t_k}\right)^{\frac{\nu}{r}+\frac{\delta}{2}}
c_{ji}^\delta c_{ik}^\delta \frac{1}{(1+t_i|x_{j,m}-x_{k,n}|_B)^{\frac{\nu}{r}+\delta}}\\
\le&\frac{C}{(1+t_j|x_{j,m}-x_{k,n}|_B)^{\frac{\nu}{r}+\delta}}\sum_{i:t_i<t_j}C\left(\frac{t_i}{t_j}\right)^{\frac{\delta}{2}-\frac{\nu}{r}-\delta}
\left(\frac{t_i}{t_k}\right)^{\frac{\nu}{r}+\frac{\delta}{2}}c_{ji}^\delta c_{ik}^\delta \\
\le&C\left(\frac{t_j}{t_k}\right)^{\frac{\nu}{r}+\frac{\delta}{2}}c_{jk}^{\delta/2}\frac{1}{(1+t_j|x_{j,m}-x_{k,n}|_B)^{\frac{\nu}{r}+\delta}}.
\end{align*}
In the case $t_j>t_k$, we observe that
$w^{s,\delta}_{(j,m)(k,n)}=w^{2\nu/r-s-\nu,\delta}_{(k,n)(j,m)}$, so
applying the first case to $w^{2\nu/r-s-\nu,\delta}_{(k,n)(j,m)}$
proves the proposition for $t_j>t_k$. \\
\end{proof}

The following two technical lemmas are used for the proof of Lemma \ref{poison}.
\begin{mylemma}\label{life1}%
Assume that $t_j\le t_k$, $N>\nu$ and
\begin{align*}
g:=\sum_{l\in\Z^d}\frac{1}{(1+\min(t_j,t_i)|x_{j,m}-x_{i,l}|_B)^{N}}\frac{1}{(1+\min(t_k,t_i)|x_{k,n}-x_{i,l}|_B)^{N}},
\end{align*}
where we have used the notation from Definition \ref{def:almostdiagonal}.
We then have
\begin{align*}
g\le
\frac{C}{(1+\min(t_j,t_i)|x_{j,m}-x_{k,n}|_B)^{N}}\max\left(\frac{t_i}{t_k},1\right)^\nu.
\end{align*}

\end{mylemma}
\begin{proof}
Note that from Lemma \ref{rock} with $r=1$ and $s_{k,n}=1$, it
follows that
\begin{align}\label{ballroom}
\sum_{l\in\Z^d}\frac{1}{(1+\min(t_k,t_i)|x_{k,n}-x_{i,l}|_B)^{N}}\le
C\max\left(\frac{t_i}{t_k},1\right)^\nu.
\end{align}
We first consider the case $\min(t_j,t_i)|x_{j,m}-x_{k,n}|_B\le 1$
which gives
\begin{align*}
g&\le \sum_{l\in\Z^d}\frac{1}{(1+\min(t_k,t_i)|x_{k,n}-x_{i,l}|_B)^{N}}\\
&\le C\max\left(\frac{t_i}{t_k},1\right)^\nu\\
&\le
\frac{C}{(1+\min(t_j,t_i)|x_{j,m}-x_{k,n}|_B)^{N}}\max\left(\frac{t_i}{t_k},1\right)^\nu.
\end{align*}
For the case $\min(t_j,t_i)|x_{j,m}-x_{k,n}|_B> 1$ we split the sum
into
\begin{align*}
A=\big\{l\in\Z^d:|x_{j,m}-x_{i,l}|_B<\tfrac{1}{2K}|x_{j,m}-x_{k,n}|_B\big\}
\end{align*}
and its complement. For $A^c$ we have
\begin{align*}
\frac{1}{(1+\min(t_j,t_i)|x_{j,m}-x_{i,l}|_B)^{N}}\le
\frac{(2K)^N}{(1+\min(t_j,t_i)|x_{j,m}-x_{k,n}|_B)^{N}},
\end{align*}
and by using \eqref{ballroom}, the desired estimate follows. For
$l\in A$, we notice that
$|x_{k,n}-x_{i,l}|_B>\tfrac{1}{2K}|x_{j,m}-x_{k,n}|_B$ and get
\begin{align}
(1+\min(t_k,t_i)|&x_{k,n}-x_{i,l}|_B)^{-N}\notag \\
&\le\left(1+\tfrac{1}{2K}\min(t_j,t_i)|x_{j,m}-x_{k,n}|_B\frac{\min(t_k,t_i)}{\min(t_j,t_i)}\right)^{-N}\notag \\
&\le
\frac{C}{(1+\min(t_j,t_i)|x_{j,m}-x_{k,n}|_B)^{N}}\left(\frac{\min(t_j,t_i)}{\min(t_k,t_i)}\right)^\nu.
\label{youth}
\end{align}
Next, by using \eqref{ballroom} with $j$ instead of
$k$ we get
\begin{align}
\sum_{l\in\Z^d}\frac{1}{(1+\min(t_j,t_i)|x_{j,m}-x_{i,l}|_B)^{N}}\le
C\max\left(\frac{t_i}{t_j},1\right)^\nu\label{today}.
\end{align}
The lemma follows by combining \eqref{youth} and \eqref{today}.\\
\end{proof}
Finally, we also used the following estimate in frequency space to prove Proposition \ref{poison}.
\begin{mylemma}\label{life2}%
Let $\delta>0$ and $0<r\le 1$. We then have
\begin{equation*}
h:=\sum_{i\in J} c^\delta_{ji} c^{\delta}_{ik} \le C c^{\delta/2}_{jk},
\end{equation*}
where
\begin{equation*}
c_{jk}^\delta:=\min\bigg(\bigg(\frac{t_j}{t_k}\bigg)^{\frac{\nu}{r}+\delta},\bigg(\frac{t_k}{t_j}\bigg)^{\delta}\bigg)
(1+\max(t_k,t_j)^{-1}|\xi_k-\xi_j|_{\mathbf{a}})^{-\frac{\nu}{r}-\delta}, \quad j,k\in J,
\end{equation*}
with the notation from Definition \ref{def:almostdiagonal}.
\end{mylemma}
\begin{proof}
Without loss of generality assume that $r=1$. We
will begin with assuming that $t_j\le t_k$. Furthermore, if
$t_k^{-1}|\xi_j-\xi_k|_{\mathbf{a}} \le \rho_0$ we have $t_k/t_j \le R_0$ by using that $\tilde{h}$ is moderate. Combining this with Lemma \ref{pinball} gives
\begin{equation*}
h \le \sum_{i\in J} c_{ik}^\delta \le C_1 \le C_2 c_{jk}^{\delta}.
\end{equation*}
In the other case, $t_k^{-1}|\xi_j-\xi_k|_{\mathbf{a}} > \rho_0$, we split the
sum into
\begin{equation*}
A=\{i\in J:|\xi_j-\xi_i|_{\mathbf{a}}<\tfrac{1}{2K}|\xi_j-\xi_k|_{\mathbf{a}}\}
\end{equation*}
and its complement. For $i\in A^c$ and $t_i\ge t_k \ge t_j$ we
have
\begin{align*}
h\le& C\sum_{\substack{i\in A^c\\i:t_i\ge t_k}}
\bigg(\frac{t_j}{t_i}\bigg)^{\nu+\delta}
(1+t_i^{-1}|\xi_j-\xi_k|_{\mathbf{a}})^{-\nu-\delta}c_{ik}^\delta\\
\le& C\bigg(\frac{t_j}{t_k}\bigg)^{\nu+\delta}
(1+t_k^{-1}|\xi_j-\xi_k|_{\mathbf{a}})^{-\nu-\delta}\sum_{\substack{i\in
A^c\\i:t_i\ge t_k}}c_{ik}^\delta\\
\le& C c_{jk}^\delta
\end{align*}
and similarly for $t_k>t_i\ge t_j$. For $t_k\ge t_j> t_i$ we get
\begin{align*}
h\le& C\sum_{\substack{i\in A^c\\i:t_i< t_j}}
\bigg(\frac{t_i}{t_j}\bigg)^{\delta}(1+t_j^{-1}|\xi_j-\xi_k|_{\mathbf{a}})^{-\nu-\delta}c_{ik}^\delta\\
\le&C\bigg(\frac{t_j}{t_k}\bigg)^{\nu+\delta}(1+t_k^{-1}|\xi_j-\xi_k|_{\mathbf{a}})^{-\nu-\delta}
\sum_{\substack{i\in A^c\\i:t_i< t_j}}c_{ik}^\delta\\
\le& C c_{jk}^\delta.
\end{align*}
Finally, when $i\in A$ we have
$|\xi_i-\xi_k|_{\mathbf{a}}>\tfrac{1}{2K}|\xi_j-\xi_k|_{\mathbf{a}}$ which for $t_i\ge
t_k \ge t_j$ gives
\begin{align*}
h\le& C\sum_{\substack{i\in A\\i:t_i\ge t_k}}
\bigg(\frac{t_k}{t_i}\bigg)^{\delta}\bigg(\frac{t_j}{t_i}\bigg)^{\nu+\delta}
(1+t_i^{-1}|\xi_j-\xi_i|_{\mathbf{a}})^{-\nu-\frac{\delta}{2}}
(1+t_i^{-1}|\xi_j-\xi_k|_{\mathbf{a}})^{-\nu-\frac{\delta}{2}} \\
\le& C\bigg(\frac{t_j}{t_k}\bigg)^{\nu+\frac{\delta}{2}}
(1+t_k^{-1}|\xi_j-\xi_k|_{\mathbf{a}})^{-\nu-\frac{\delta}{2}}
\sum_{\substack{i\in A\\i:t_i\ge t_k}}
\bigg(\frac{t_j}{t_i}\bigg)^{\frac{\delta}{2}}
(1+t_i^{-1}|\xi_j-\xi_i|_{\mathbf{a}})^{-\nu-\frac{\delta}{2}} \\
\le& C c_{jk}^{\delta/2}.
\end{align*}
For $t_k>t_i\ge t_j$ and $t_k\ge t_j> t_i$ the argument can be
repeated in a similar way which proves the lemma when $t_k\ge
t_j$. For $t_k<t_j$, it suffices to use that $
c_{jk}^\delta=(t_j/t_k)^\nu c_{kj}^\delta$, and we get
\begin{align*}
h=\sum_{i\in J} \bigg(\frac{t_j}{t_i}\bigg)^{\nu}c_{ij}^\delta
\bigg(\frac{t_i}{t_k}\bigg)^{\nu}c_{ki}^\delta\le C
\bigg(\frac{t_j}{t_k}\bigg)^{\nu}c_{kj}^{\delta/2}=c_{jk}^{\delta/2}.
\end{align*}
\end{proof}

%\bibliographystyle{abbrv}
%\bibliography{decomp_bib}

\begin{thebibliography}{10}
	
	\bibitem{AN19}
	Z.~Al-Jawahri and M.~Nielsen.
	\newblock On homogeneous decomposition spaces and associated decompositions of
	distribution spaces.
	\newblock {\em Mathematische Nachrichten}, 292(12):2496--2521, 2019.
	
	\bibitem{Nielsen}
	L.~Borup and M.~Nielsen.
	\newblock Frame decomposition of decomposition spaces.
	\newblock {\em J. Fourier Anal. Appl.}, 13(1):39--70, 2007.
	
	\bibitem{BN08}
	L.~Borup and M.~Nielsen.
	\newblock On anisotropic {T}riebel-{L}izorkin type spaces, with applications to
	the study of pseudo-differential operators.
	\newblock {\em J. Funct. Spaces Appl.}, 6(2):107--154, 2008.
	
	\bibitem{MR2179611}
	M.~Bownik.
	\newblock Atomic and molecular decompositions of anisotropic {B}esov spaces.
	\newblock {\em Math. Z.}, 250(3):539--571, 2005.
	
	\bibitem{BH05}
	M.~Bownik and K.-P. Ho.
	\newblock Atomic and molecular decompositions of anisotropic
	{T}riebel-{L}izorkin spaces.
	\newblock {\em Trans. Amer. Math. Soc.}, 358(4):1469--1510, 2006.
	
	\bibitem{MR3707993}
	G.~Cleanthous, A.~G. Georgiadis, and M.~Nielsen.
	\newblock Anisotropic mixed-norm {H}ardy spaces.
	\newblock {\em J. Geom. Anal.}, 27(4):2758--2787, 2017.
	
	\bibitem{MR3981281}
	G.~Cleanthous, A.~G. Georgiadis, and M.~Nielsen.
	\newblock Molecular decomposition of anisotropic homogeneous mixed-norm spaces
	with applications to the boundedness of operators.
	\newblock {\em Appl. Comput. Harmon. Anal.}, 47(2):447--480, 2019.
	
	\bibitem{MR3016517}
	S.~Dahlke, S.~H\"auser, G.~Steidl, and G.~Teschke.
	\newblock Shearlet coorbit spaces: traces and embeddings in higher dimensions.
	\newblock {\em Monatsh. Math.}, 169(1):15--32, 2013.
	
	\bibitem{MR3303346}
	S.~Dekel, G.~Kerkyacharian, G.~Kyriazis, and P.~Petrushev.
	\newblock Compactly supported frames for spaces of distributions associated
	with nonnegative self-adjoint operators.
	\newblock {\em Studia Math.}, 225(2):115--163, 2014.
	
	\bibitem{MR910054}
	H.~G. Feichtinger.
	\newblock Banach spaces of distributions defined by decomposition methods.
	{II}.
	\newblock {\em Math. Nachr.}, 132:207--237, 1987.
	
	\bibitem{MR809337}
	H.~G. Feichtinger and P.~Gr\"obner.
	\newblock Banach spaces of distributions defined by decomposition methods. {I}.
	\newblock {\em Math. Nachr.}, 123:97--120, 1985.
	
	\bibitem{MR1070037}
	M.~Frazier and B.~Jawerth.
	\newblock A discrete transform and decompositions of distribution spaces.
	\newblock {\em J. Funct. Anal.}, 93(1):34--170, 1990.
	
	\bibitem{MR3345605}
	H.~F\"uhr and F.~Voigtlaender.
	\newblock Wavelet coorbit spaces viewed as decomposition spaces.
	\newblock {\em J. Funct. Anal.}, 269(1):80--154, 2015.
	
	\bibitem{MR4029179}
	A.~G. Georgiadis, G.~Kerkyacharian, G.~Kyriazis, and P.~Petrushev.
	\newblock Atomic and molecular decomposition of homogeneous spaces of
	distributions associated to non-negative self-adjoint operators.
	\newblock {\em J. Fourier Anal. Appl.}, 25(6):3259--3309, 2019.
	
	\bibitem{MR3243741}
	L.~Grafakos.
	\newblock {\em Modern {F}ourier analysis}, volume 250 of {\em Graduate Texts in
		Mathematics}.
	\newblock Springer, New York, third edition, 2014.
	
	\bibitem{MR3969420}
	L.~Huang, J.~Liu, D.~Yang, and W.~Yuan.
	\newblock Atomic and {L}ittlewood-{P}aley characterizations of anisotropic
	mixed-norm {H}ardy spaces and their applications.
	\newblock {\em J. Geom. Anal.}, 29(3):1991--2067, 2019.
	
	\bibitem{MR1862566}
	G.~Kyriazis and P.~Petrushev.
	\newblock New bases for {T}riebel-{L}izorkin and {B}esov spaces.
	\newblock {\em Trans. Amer. Math. Soc.}, 354(2):749--776, 2002.
	
	\bibitem{MR2204289}
	G.~Kyriazis and P.~Petrushev.
	\newblock On the construction of frames for {T}riebel-{L}izorkin and {B}esov
	spaces.
	\newblock {\em Proc. Amer. Math. Soc.}, 134(6):1759--1770, 2006.
	
	\bibitem{MR2891260}
	G.~Kyriazis and P.~Petrushev.
	\newblock ``{C}ompactly'' supported frames for spaces of distributions on the
	ball.
	\newblock {\em Monatsh. Math.}, 165(3-4):365--391, 2012.
	
	\bibitem{MR3048591}
	D.~Labate, L.~Mantovani, and P.~Negi.
	\newblock Shearlet smoothness spaces.
	\newblock {\em J. Fourier Anal. Appl.}, 19(3):577--611, 2013.
	
	\bibitem{NR12}
	M.~Nielsen and K.~N. Rasmussen.
	\newblock Compactly supported frames for decomposition spaces.
	\newblock {\em J. Fourier Anal. Appl.}, 18(1):87--117, 2012.
	
	\bibitem{MR1232192}
	E.~M. Stein.
	\newblock {\em Harmonic analysis: real-variable methods, orthogonality, and
		oscillatory integrals}, volume~43 of {\em Princeton Mathematical Series}.
	\newblock Princeton University Press, Princeton, NJ, 1993.
	\newblock With the assistance of Timothy S. Murphy, Monographs in Harmonic
	Analysis, III.
	
	\bibitem{Stein}
	E.~M. Stein and S.~Wainger.
	\newblock Problems in harmonic analysis related to curvature.
	\newblock {\em Bull. Amer. Math. Soc.}, 84(6):1239--1295, 1978.
	
	\bibitem{MR725159}
	H.~Triebel.
	\newblock Modulation spaces on the {E}uclidean {$n$}-space.
	\newblock {\em Z. Anal. Anwendungen}, 2(5):443--457, 1983.
	
	\bibitem{MR3409094}
	H.~Triebel.
	\newblock {\em Tempered homogeneous function spaces}.
	\newblock EMS Series of Lectures in Mathematics. European Mathematical Society
	(EMS), Z\"urich, 2015.
	
	\bibitem{MR3682619}
	H.~Triebel.
	\newblock Tempered homogeneous function spaces, {II}.
	\newblock In {\em Functional analysis, harmonic analysis, and image processing:
		a collection of papers in honor of {B}j\"orn {J}awerth}, volume 693 of {\em
		Contemp. Math.}, pages 331--361. Amer. Math. Soc., Providence, RI, 2017.
	
	\bibitem{2016arXiv160102201V}
	F.~{Voigtlaender}.
	\newblock {Embeddings of Decomposition Spaces into Sobolev and BV Spaces}.
	\newblock {\em ArXiv e-prints}, Jan. 2016.
	
\end{thebibliography}
\end{document}